\documentclass[final, 3p, times,review]{elsarticle}
\usepackage{color}
\usepackage{hyperref}
\usepackage{amsmath,amsfonts,amsthm,amssymb}
\usepackage{algorithm2e}
\usepackage{bm,dsfont}
\usepackage{booktabs}
\usepackage{graphicx}
\usepackage{array}
\usepackage{xcolor}
\usepackage{sidecap}
\usepackage{listings}
\usepackage{times}
\usepackage{url}
\usepackage{stmaryrd}
\usepackage{ulem}
\newtheorem{theorem}{Theorem}[section]
\newtheorem{remark}{Remark}
\newtheorem{lemma}[theorem]{Lemma}

\newtheorem{assumption}{Assumption}[section]

\newcommand{\norm}[1]{\left\lVert#1\right\rVert}
\newcommand{\R}{\mathbb{R}}

\begin{document}
\begin{frontmatter}
\title{Convergence of physics-informed neural networks applied to linear second-order elliptic interface problems}
\author[cas,ucas]{Sidi Wu}
\author[cas,ucas]{Aiqing Zhu}
\author[cas,ucas]{Yifa Tang}
\author[cas,ucas]{Benzhuo Lu\corref{cor}}
\cortext[cor]{Corresponding author: bzlu@lsec.cc.ac.cn (Benzhuo Lu)}

\address[cas]{State Key Laboratory of Scientific and Engineering Computing, National Center for Mathematics and Interdisciplinary Sciences, Academy of Mathematics and Systems Science, Chinese Academy of Sciences, Beijing 100190, China}
\address[ucas]{School of Mathematical Sciences, University of Chinese Academy of Sciences, Beijing 100049, China}

\begin{abstract}
With the remarkable empirical success of neural networks across diverse scientific disciplines, rigorous error and convergence analysis are also being developed and enriched. However, there has been little theoretical work focusing on neural networks in solving interface problems. In this paper, we perform a convergence analysis of physics-informed neural networks (PINNs) for solving second-order elliptic interface problems. Specifically, we consider PINNs with domain decomposition technologies and introduce gradient-enhanced strategies on the interfaces to deal with boundary and interface jump conditions. It is shown that the neural network sequence obtained by minimizing a Lipschitz regularized loss function converges to the unique solution to the interface problem in $H^2$ as the number of samples increases. Numerical experiments are provided to demonstrate our theoretical analysis.
\end{abstract}

\begin{keyword}
Elliptic interface problems; Generalization  errors; Convergence analysis; Neural networks.
\end{keyword}
\end{frontmatter}

\section{Introduction}
Deep learning in the form of deep neural networks (DNNs) has been effectively used in diverse scientific disciplines beyond its traditional applications. In particular, thanks to their potential nonlinear approximation power \cite{pinkus1999approximation,darbon2021some,zhu2022approximation}, DNNs are being exploited to construct alternative approaches for solving partial differential equations (PDEs), e.g., the deep Ritz method (DRM) \cite{weinan2017deep} and physics-informed neural networks (PINNs) \cite{raissi2019physics}. The key idea of these methods is to reformulate the solution to a PDE with a closed-form expression in the form of a neural network, the parameters of which are obtained by minimizing a physics-informed loss given by the corresponding PDE. The original works on the use of neural networks to solve PDEs were proposed in the $1990s$ \cite{lee1990neural,lagaris1998artificial}, and this idea has recently been revisited with the renaissance of neural networks and the development of deep learning techniques; see e.g., \cite{rudd2015constrained,pang2019fpinns,zang2020weak,wang2021learning,zhu2019physics} and references therein.

Elliptic interface problems are a widespread class of problems in scientific computing with many applications across diverse fields; see e.g. \cite{hou1997hybrid,lu2008recent,liu2020moment,wang2019petrov}. There are many accurate and efficient numerical methods in the literature for interface problems, such as the finite element method (FEM) \cite{chen2020bilinear,ji2018finite}, the discontinuous Galerkin method (DG) \cite{massjung2012unfitted,saye2019efficient}, the immersed interface method (IIM) \cite{leveque1994immersed,hou2005numerical}, the immersed boundary method (IBM) \cite{peskin2002immersed}, the boundary element method (BEM) \cite{lu2008recent}, and the voronoi interface method (VIM) \cite{guittet2015solving}. In the last few decades, the numerical methods for solving interface problems have reached a certain maturity and made satisfactory progress. However, the above-mentioned methods usually require either a body-fitted or unfitted mesh to treat the interface problems, and the main difficulty lies in the body-fitted mesh generation or in the technique designed to dissect the intersecting geometry of the interface and properly discretize interface conditions. Interface problems are still challenging due to the low global regularity and irregular geometry of interfaces.

In recent years, many efforts have been made to use neural networks to solve interface problems since these methods are meshfree and can take advantage of deep learning techniques such as automatic differentiation and GPU acceleration.  In particular, neural network-based approaches exhibit notable advantages in treating high-dimensional problems, inverse problems, and simultaneously solving parametric PDE problems that involve learning the solution operator (operator learning), which issues also exist in interface problems. In addition, the use of multiple neural networks based on the domain decomposition method (DDM) has attracted increasing attention as they are more accurate and flexible in dealing with the interface and have shown remarkable success in various interface problems \cite{jagtap2020extended,li2020deep,he2020mesh,wu2022interface}. This idea is further studied from the numerical aspect in our previous work \cite{wu2022interface}, where the proposed interfaced neural networks are able to balance the interplay between different terms in the composite loss function and improve the performance in terms of accuracy and robustness. The above-mentioned works focus on obtaining empirical results, whereas we focus on theoretical aspects such as the  convergence of PINNs for solving interface problems in this paper.

Along with the remarkable empirical achievements of deep learning methods, rigorous error and convergence analysis are also being developed and enriched. In previous work \cite{shin2020on}, the H\"older continuity constant was used to obtain the generalization analysis of PINNs in the case of linear second-order elliptic and parabolic type PDEs. \cite{mishra2021estimates, mishra2022estimates} used quadrature points in the formulation of the loss and carried out an a-posteriori-type generalization error analysis of PINNs for both forward and inverse problems. \cite{shin2020error} studied linear PDEs and proved both a priori and posterior estimates for PINNs and variational PINNs in Sobolev spaces. \cite{hu2021extended} provided a theoretical understanding of the generalization abilities of PINNs and Extended PINNs (XPINNs) \cite{jagtap2020extended}. \cite{luo2020two} derived an a priori generalization estimate for a class of second-order linear PDEs in the context of two-layer neural networks by assuming that the exact solutions of PDEs belong to a Barron-type space \cite{barron1993universal}.  \cite{jiao2021convergence} provided a nonasymptotic convergence rate of PINNs with $\text{ReLU}^3$ networks for the second linear elliptic equation. For high-dimensional PDEs,  \cite{lu2021priori} derived a priori and dimension explicit generalization error estimates for the DRM under the assumption that the solutions of the PDEs lie in the spectral Barron space, and \cite{berner2020analysis} provided an analysis of the generalization error for linear Kolmogorov equations by using tools from statistical learning theory and covering number estimates of neural network hypothesis classes.

However, the majority of existing theoretical works are limited to differential equations with continuous coefficients; much less is known about the convergence of PINNs in solving interface problems, where the interaction at the interface introduces additional analytical challenges. In particular, it is reasonable to make the assumption that the network satisfies the boundary conditions for elliptic problems (see, for example, Theorem 3.4 in \cite{shin2020on}), since there are several approaches \cite{lagaris1998artificial,lagari2020systematic} to forcing neural networks to obey the boundary conditions intrinsically. But such approaches cannot be applied to interface jump conditions.
When considering the convergence of interface problems, we are inevitably faced with the challenge of estimating errors caused by interface losses.

By extending the convergence results in \cite{shin2020on} to elliptic interface problems, we provide a convergence theory for DDM-based PINNs to solve linear second-order elliptic interface problems. In this work, to deal with the error caused by the non-zero interface and boundary losses, we introduce a gradient enhancement strategy on interfaces inspired by \cite{yu2022gradient}, where the gradient information from the residual of the boundary and interface jump conditions is embedded into the loss function. Following the work of Shin $et \ al.$ \cite{shin2020on}, we construct a specific Lipschitz regularization loss tailored for elliptic interface problems to quantify the generalization of PINN. Finally, we prove that the sequence of minimizers of the designed regularized loss function converges to the unique solution to the interface problem in $H^2$ under some reasonable assumptions. To the best of our knowledge, this is the first theoretical work that proves the convergence of neural network methods for solving elliptic interface problems. The main contributions of our work can be summarized as follows:
\begin{itemize}
    \item We introduce gradient-enhanced strategies on interfaces to estimate the error caused by non-zero interface and boundary losses.
    \item We first provide the convergence analysis of PINNs in solving elliptic interface problems.
    \item We present several numerical experiments to validate the theoretical analysis.
\end{itemize}

The rest of this paper is organized as follows. In Section \ref{Preliminaries}, some preliminaries, including notations and background knowledge of neural networks and interface problems, are introduced. In Section \ref{NN4interface problem}, we briefly introduce the algorithm of PINN for solving elliptic interface problems and present the gradient-enhanced strategies on the interfaces. In Section \ref{main results}, we present a convergence analysis, the proof of which is presented in Section \ref{proofs}. Numerical experiments are performed in Section \ref{numerical examples} to validate the theoretical analysis. Finally, we conclude the paper in Section \ref{summary}.

\section{Preliminaries\label{Preliminaries}}
\subsection{Notations}
We first introduce some notations. Let $\mathbf{x}=(x_1,\ldots,x_d)$ be a point in $\mathbb{R}^d$ ($d\geq 2$) and $\mathcal{U} \subset \mathbb{R}^d$ be an open set. Let $C(\mathcal{U}) = \left\{ f: \mathcal{U} \to \mathbb{R}^{d'} | f \text{ is continuous} \right\}$ denotes the space of continuous functions. Let $\mathbb{Z}_+^d$ denotes the lattice of $d$-dimensional nonnegative integers. For $\mathbf{k} = (k_1, \dots, k_d) \in \mathbb{Z}_+^d$, we set $|\mathbf{k}|:= k_1+\dots +k_d$, and $$D^{\mathbf{k}} \:= \frac{\partial^{|\mathbf{k}|}}{\partial x_1^{k_1}\dots \partial x_d^{k_d}}.$$
For a positive integer $k$, we define
\begin{equation*}
C^k(\mathcal{U}):=\{f: D^{\mathbf{k}}f\in C(\mathcal{U}) \text{ for all}\ |\mathbf{k}|\leq k \}.
\end{equation*}
Subsequently, we denote $\big[\mu \big]_{\mathcal{U}}$ to be the Lipshcitz constant of $\mu$ on $\mathcal{U}$, i.e.,
\begin{equation*}
\big[\mu \big]_{\mathcal{U}} = \sup_{\mathbf{x}, \mathbf{y} \in \mathcal{U}, \mathbf{x}\neq \mathbf{y}} \frac{\norm{\mu(\mathbf{x}) - \mu(\mathbf{y})}_{\infty}}{\norm{\mathbf{x}-\mathbf{y}}_{\infty}}.
\end{equation*}
In order to distinguish from the $k$-times continuously differentiable function space $C^{k}$, we denote $C^{k,L}(\mathcal{U})$ to be the collection of functions in $C^{k}$ whose derivatives of order $k$ are Lipschitz continuous.

Following \cite{lions2012non}, we present the definition of $H^s(E)$ on the boundary $E$. Here, we suppose $E$ is a $(d-1)$-dimensional smooth manifold, i.e., there exists a collection of charts $\{(\mathcal{V}_i,\phi_i)\ | \ i\in I\}$ such that $\{\mathcal{V}_i\}_{i\in I}$ is a collection of open sets on $E$ and covers $E$ (i.e., $E=\cup_{i\in I}\mathcal{V}_i$), and such
that $\phi_i$ is homeomorphism from $\mathcal{V}_i$ to an open subset $\mathcal{V}_i':=\phi_i(\mathcal{V}_i)$ of $\mathbb{R}^{d-1}$ for all $i \in I$ and the transition map $\phi_i\circ\phi_j^{-1}: \phi_j(\mathcal{V}_i\cap\mathcal{V}_j)\rightarrow\phi_i(\mathcal{V}_i\cap\mathcal{V}_j)$ is an infinitely differentiable mapping when $\mathcal{V}_i\cap\mathcal{V}_j\neq\emptyset$ for all $i, j \in I$ . Let $\{\eta_i\}_{i\in I}$ be a partition of unity on $E$ with compact support in $\mathcal{V}_i$ such
that $\sum_i\eta_i(\mathbf{x})=1$ for all $\mathbf{x}\in E$ and $\eta_i$ are infinitely differentiable. Then, if $u$ is a function on $E$, we can decompose $u=\sum_i(\eta_iu)$, and define
\begin{equation*}
\phi_i^*(\eta_iu)(\boldsymbol{\xi})=(\eta_iu)(\phi_i^{-1}(\boldsymbol{\xi})),\ \boldsymbol{\xi}\in\mathcal{V}_i'.
\end{equation*}
Finally, we define
\begin{equation*}\label{eq: Hs}
H^s(E)=\{u\ | \ \phi_i^*(\eta_iu)\in H^s(\mathcal{V}'_i), \forall i\in I\},
\end{equation*}
with norm
\begin{equation}\label{eq: Hs norm}
\norm{u}_{H^s(E)}=\left(\sum_{i}\norm{\phi_i^*(\eta_iu)}^2_{H^s(\mathcal{V}'_i)}\right)^\frac{1}{2}.
\end{equation}
It is easy to verify that $H^s(E)$ is a Hilbert space and that the different norms \eqref{eq: Hs norm} are equivalent. We refer the readers to \cite{lions2012non} for more details.

For given $\{\mathcal{V}_i,\phi_i,\eta_i\}_{i\in I}$ and $f\in C^2(E)$, we define that if dimension $d=2$,
\begin{equation*}
D_Ef=\sum_{i}\frac{\partial \phi_i^*(\eta_iu)}{\partial \xi_1}, \quad D_E^2f=\sum_{i}\frac{\partial^2 \phi_i^*(\eta_iu)}{\partial \xi_1^2},
\end{equation*}
and if dimension $d=3$,
\begin{equation*}
D_Ef=\left(\sum_{i}\frac{\partial \phi_i^*(\eta_iu)}{\partial \xi_1},\sum_{i}\frac{\partial \phi_i^*(\eta_iu)}{\partial \xi_2}\right), D_E^2f=\left(\sum_{i}\frac{\partial^2 \phi_i^*(\eta_iu)}{\partial \xi_1^2},\sum_{i}\frac{\partial^2 \phi_i^*(\eta_iu)}{\partial \xi_1\partial \xi_2},\sum_{i}\frac{\partial^2 \phi_i^*(\eta_iu)}{\partial \xi_2^2}\right),
\end{equation*}
where $\xi_j$ is the $j$-th component of $\boldsymbol{\xi}$.

\subsection{Neural networks}
In addition, we introduce the employed network architecture, i.e., the feed-forward neural network (FNN), in this paper. Mathematically, an $N$-layer FNN is a nested composition of sequential linear functions and nonlinear activation functions, which takes the form
\begin{equation*}\label{eq:nn}
    \begin{aligned}
    &\mathbf{s}_{i}=f_i(\mathbf{s}_{i-1}):=\sigma (\mathbf{W}_i\mathbf{s}_{i-1}+\mathbf{b}_i) ,\ \text{for}\ i=1, \cdots, N-1, \\
    &\mathbf{s}_{N}=f_N(\mathbf{s}_{N-1}):= \mathbf{W}_N\mathbf{s}_{N-1}+\mathbf{b}_N,
    \end{aligned}
\end{equation*}
where $\mathbf{s}_0=\mathbf{x}\in \mathbb{R}^{d_{in}}$ is the input variable, $\mathbf{s}_i\in \mathbb{R}^{d_i}$ denotes the output of the $i$-th hidden layer,  $\mathbf{s}_{N} \in \mathbb{R}^{d_{out}}$ is the corresponding output, and $\mathbf{W}_i\in \mathbb{R}^{d_{i+1}\times d_i}$ and $\mathbf{b}_i \in \mathbb{R}^{d_{i+1}}$ are trainable parameters. $\sigma: \R\rightarrow \R$ is the nonlinear activation function applied element-wise to a vector. Popular examples include the rectified linear unit (ReLU) $\text{ReLU}(z) = \max(0, z)$, the logistic sigmoid $\text{Sig}(z) = 1/(1 + e^{-z})$ and the hyperbolic tangent $\text{tanh}(z)=(e^{z}-e^{-z})/(e^{z}+e^{-z})$. Equipped with those definitions, the FNN representation of a continuous function can be viewed as
\begin{equation}
\label{eq: nn expression}
\mathcal{NN}(\mathbf{x})=f_N\circ \cdots \circ f_{1}(\mathbf{x}).
\end{equation}
Furthermore, we denote all the trainable parameters (e.g., $\mathbf{W}_i$, $\mathbf{b}_i$) in \eqref{eq: nn expression} as $\boldsymbol{\theta}\in\Theta$, where $\boldsymbol{\theta}$ is a high-dimensional vector and $\Theta$ is the space of $\boldsymbol{\theta}$. Given a network architecture $\bm{\overrightarrow{n}}$ (e.g., the number of layers and the width of each hidden layer), we denote the set of all expressible functions (hypothesis space) as
\begin{equation}\label{def:fcnn}
\mathcal{H}_{\bm{\overrightarrow{n}}}^\text{NN}=\{\mathcal{NN}(\cdot;\bm{\overrightarrow{n}},\boldsymbol{\theta}): \mathbb{R}^{d_{in}}\rightarrow\mathbb{R}^{d_{out}}\big|\boldsymbol{\theta}\in\Theta\}.
\end{equation}

\subsection{Elliptic interface problems}

Let $\Omega=\Omega_1\cup\Omega_2$ be a bounded domain in $\mathbb{R}^d$ with smooth boundary $\partial \Omega$. Let $\Omega_1\subset\Omega$ be an open domain with smooth boundary $\Gamma=\partial\Omega_1\subset\Omega$ and $\Omega_2=\Omega\setminus\Omega_1$ (see Fig. \ref{fig: doamin} for an illustration). We consider the following linear second-order elliptic interface problem
\begin{subequations}  \label{eq:interface problem}
\begin{align}
-\nabla\cdot(a_i\nabla u)+b_i u &= f_i,\quad \text{in} \ \Omega_i, \ i=1,2, \label{eq:1A}\\
 \llbracket a\nabla u\cdot \mathbf{n} \rrbracket &=\psi, \quad  \text{on} \ \Gamma \label{eq:1B}, \\
\llbracket u\rrbracket &=\varphi,\quad  \text{on} \ \Gamma \label{eq:1C},\\
u&=g,\quad \text{on} \ \partial\Omega \label{eq:1D},
\end{align}
\end{subequations}
where $\llbracket\mu\rrbracket:= \mu |_{\Omega_2}-\mu |_{\Omega_1}$ denotes the jump of a quantity $\mu$ across $\Gamma$, and $\mathbf{n}$ denotes the unit outward normal of $\Omega_1$. The coefficients
\begin{equation*}
a(\mathbf{x}) = \left\{\begin{aligned}
&a_1(\mathbf{x}),\quad \text{if}\ \mathbf{x} \in \Omega_1\\
&a_2(\mathbf{x}),\quad \text{if}\ \mathbf{x} \in \Omega_2\\
\end{aligned}\right., \quad
b(\mathbf{x}) = \left\{\begin{aligned}
&b_1(\mathbf{x}),\quad \text{if}\ \mathbf{x} \in \Omega_1\\
&b_2(\mathbf{x}),\quad \text{if}\ \mathbf{x} \in \Omega_2\\
\end{aligned}\right.
\end{equation*}
are piecewise spatial functions. The unknown part of this problem is the exact solution $u^*$, while others are given in advance.
\begin{figure}[htbp]
	\centering
	\scalebox{0.9}{\includegraphics{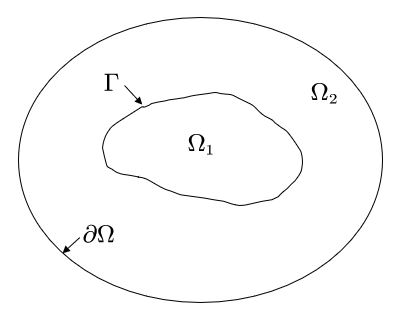}}
	\caption{A schematic view of the geometry description.}
\label{fig: doamin}
\end{figure}

\section{Neural network methods for linear second-order elliptic interface problems}\label{NN4interface problem}

In this section, we present a brief overview of physics-informed neural networks (PINNs) \cite{raissi2019physics} for solving linear second-order elliptic interface problems. Since the physical solutions to elliptic interface problems are usually non-smooth or even discontinuous across the interface, it is natural to use domain decomposition methods (DDMs) in the PINN framework \cite{jagtap2020extended,li2020deep,he2020mesh,wu2022interface}. In the context of DDM-based deep learning methods, the computational domain is decomposed into several disjoint subdomains according to the interface, and the solution to the interface problem is the combination and ensemble of multiple local networks, where each of them is responsible for prediction in one subdomain.

Under the PINN framework, we approximate the latent solution to interface problems by two FNNs, i.e., $u_1(\mathbf{x},\boldsymbol{\theta}_1)\big|_{\Omega_1}$ and $u_2(\mathbf{x},\boldsymbol{\theta}_2)\big|_{\Omega_2}$. Let $\boldsymbol{\theta}=(\boldsymbol{\theta}_1, \boldsymbol{\theta}_2)$ denotes all tunable parameters of the networks (e.g., weights and biases). We then use the constraints implied by Eq. \eqref{eq:interface problem} and the boundary and interface jump conditions to train the networks. Let us denote the number of training data points by $\bm{m} = (m_{r_1}, m_{r_2}, m_{\Gamma}, m_{b})$, where $m_{r_1}, m_{r_2}, m_{\Gamma}$ and $m_b$ represent the number of training samples in $\Omega_1$, $\Omega_2$, $\Gamma$ and $\partial\Omega$, respectively.  Then, a physics-informed model can be trained by minimizing the following composite empirical loss function
\begin{equation} \label{def: vanilla PINN empirical loss}
\begin{aligned}
\text{Loss}^{\text{PINN}}_{\bm{m}}(u_1,u_2;\bm{\lambda})
:=& \lambda_{r_1}\mathcal{M}_{\Omega_1}
+ \lambda_{r_2}\mathcal{M}_{\Omega_2}
+ \lambda_{b}\mathcal{M}_{b}+\lambda_{\Gamma_D}\mathcal{M}_{\Gamma_D}  +\lambda_{\Gamma_N}\mathcal{M}_{\Gamma_N},
\end{aligned}
\end{equation}
where $\bm{\lambda} = (\lambda_{r_1}, \lambda_{r_2}, \lambda_{\Gamma} ,\lambda_{b})\geq 0$ (element-wise inequality), the loss terms $\mathcal{M}_{\Omega_1}$ and $\mathcal{M}_{\Omega_2}$ correspond to the PDE residuals (\ref{eq:1A}) in $\Omega_1$ and $\Omega_2$, $\mathcal{M}_{
b}$, $\mathcal{M}_{\Gamma_D}$ and $\mathcal{M}_{\Gamma_N}$ enforce the boundary condition (\ref{eq:1D}), interface jump conditions (\ref{eq:1C}) and (\ref{eq:1B}), respectively. For a typical interface problem \eqref{eq:interface problem}, we define
\begin{equation*}
\begin{aligned}
&\mathcal{L}_i[u_i] = -\nabla\cdot(a_i\nabla u_i)+b_i u_i, \text{with } i=1,2,
\quad
\mathcal{B}[u_2] = u_2,
\\
&\mathcal{I}_D[u_1, u_2] = u_2-u_1,
\quad
\mathcal{I}_N[u_1, u_2] = a_2\nabla u_2\cdot \mathbf{n} - a_1\nabla u_1\cdot \mathbf{n},
\end{aligned}
\end{equation*}
which can be derived by automatic differentiation \cite{baydin2017automatic}.
Then, the loss terms in $\text{Loss}_{\mathbf{m}}^{\text{PINN}}$ \eqref{def: vanilla PINN empirical loss} would take the specific form
\begin{equation*}
\begin{aligned}
&\mathcal{M}_{\Omega_1} =\frac{1}{m_{r_1}} \sum_{i=1}^{m_{r_1}} \left|\mathcal{L}_1[u_1](\mathbf{x}_{r_1}^i)- f_1(\mathbf{x}_{r_1}^i)\right|^2,
\
\mathcal{M}_{\Omega_2} =\frac{1}{m_{r_2}}\sum_{i=1}^{m_{r_2}} \left|\mathcal{L}_2[u_2](\mathbf{x}_{r_2}^i)- f_2(\mathbf{x}_{r_2}^i)\right|^2,
\\
&\mathcal{M}_{b}=\frac{1}{m_{b}}\sum_{i=1}^{m_{b}}\norm{\mathcal{B}[u_2](\mathbf{x}_{b}^i) - g(\mathbf{x}_{b}^i)}_2^2, \
\mathcal{M}_{\Gamma_D} = \frac{1}{m_{\Gamma}}\sum_{i=1}^{m_{\Gamma}}
\norm{\mathcal{I}_D[u_1, u_2](\mathbf{x}_{\Gamma}^i) - \varphi(\mathbf{x}_{\Gamma}^i)}_2^2,
\\
&\mathcal{M}_{\Gamma_N} = \frac{1}{m_{\Gamma}}\sum_{i=1}^{m_{\Gamma}}
\norm{\mathcal{I}_N[u_1, u_2](\mathbf{x}_{\Gamma}^i)- \psi(\mathbf{x}_{\Gamma}^i)}_2^2,
\end{aligned}
\end{equation*}
where $\{\mathbf{x}_{b}^i\}_{i=1}^{m_{b}}:=\mathcal{T}_{b}^{m_{b}}$ denotes the boundary data points,  $\{\mathbf{x}_{\Gamma}^i\}_{i=1}^{m_{\Gamma}}:=\mathcal{T}_{\Gamma}^{m_{\Gamma}}$ denotes the interface data points, and  $\{\mathbf{x}_{r_j}^i\}_{i=1}^{m_{r_j}}:= \mathcal{T}_{r_j}^{m_{r_j}}$ denotes the training data points that are randomly placed insider the subdomain $\Omega_j$ with $j=1,2$. Here, we suppose these four types of data sets are independently and identically (iid) sampled from probability distributions $\mu_{r_1}$, $\mu_{r_2}$, $\mu_{\Gamma}$ and $\mu_{b}$, respectively.

In PINNs, we only enforce the residual of boundary and interface jump conditions to be zero, while in this work, we introduce gradient-enhanced strategies to the PINN framework to estimate the error caused by non-zero boundary and interface losses.
Specifically, the high-order gradient information of the interface(s) is embedded into the loss function by redefining the following loss terms,
\begin{equation*}
\begin{aligned}
&\mathcal{M}_{b}=\frac{1}{m_{b}}\sum_{i=1}^{m_{b}}\norm{\mathcal{B}[u_2](\mathbf{x}_{b}^i) - \bm{g}(\mathbf{x}_{b}^i)}_2^2,\
\mathcal{M}_{\Gamma_D} = \frac{1}{m_{\Gamma}}\sum_{i=1}^{m_{\Gamma}}
\norm{\mathcal{I}_D[u_1, u_2](\mathbf{x}_{\Gamma}^i) - \bm{\varphi}(\mathbf{x}_{\Gamma}^i)}_2^2,
\\
&\mathcal{M}_{\Gamma_N} = \frac{1}{m_{\Gamma}}\sum_{i=1}^{m_{\Gamma}}
\norm{\mathcal{I}_N[u_1, u_2](\mathbf{x}_{\Gamma}^i)- \bm{\psi}(\mathbf{x}_{\Gamma}^i)}_2^2,
\end{aligned}
\end{equation*}
where
\begin{equation*}
\begin{aligned}
&\mathcal{B}[u_2] =(u_2, D_{\partial\Omega}u_2, D_{\partial\Omega}^2u_2),\quad
\mathcal{I}_D[u_1, u_2] = ( u_2-u_1 , D_\Gamma(u_2-u_1), D_\Gamma^2(u_2-u_1)),\\
&\mathcal{I}_N[u_1, u_2] = (a_2\nabla u_2\cdot \mathbf{n} - a_1\nabla u_1\cdot \mathbf{n}, D_\Gamma(a_2\nabla u_2\cdot \mathbf{n} - a_1\nabla u_1\cdot \mathbf{n})),
\end{aligned}
\end{equation*}
and
\begin{equation*}
 \bm{\varphi} = (\varphi, D_\Gamma\varphi,    D_\Gamma^2\varphi),\quad \bm{\psi} = (\psi, D_\Gamma\psi), \quad\bm{g} = (g, D_{\partial\Omega} g, D_{\partial\Omega} ^2g).
\end{equation*}
Note that since $u^*(\mathbf{x})-g(\mathbf{x})=0$ for any $\mathbf{x}$ on the boundary, we know that for any positive integer $k$, the derivative $D_{\partial\Omega}^k(u^*-g)$ is zero. The same is true for the interface residuals. Hence, it is acceptable for us to enforce the derivatives of the residual of interface(s) to be zero.
Obviously, the gradient-enhanced empirical PINN loss is an upper bound for vanilla empirical PINN loss. Unless otherwise stated, the rest of the paper discusses gradient-enhanced empirical PINN loss $\text{Loss}_{\mathbf{m}}^{\text{PINN}}$ \eqref{def: vanilla PINN empirical loss}. In addition, motivated by the upper bound \cite{shin2020on}, we consider the Lipschitz regularized loss function
\begin{equation}\label{def: GE LIP PINN empirical loss}
\begin{aligned}
\text{Loss}_{\bm{m}}(u_1,u_2;\bm{\lambda}, \bm{\lambda}^R)
:=&\text{Loss}^{\text{PINN}}_{\bm{m}}(u_1,u_2;\bm{\lambda}) + \lambda_{r_1}^R R_{r_1}(u_1) + \lambda_{r_2}^R R_{r_2}(u_2)\\ &+\lambda_{b}^R R_{b}(u_{2})+ \lambda_{\Gamma_D}^R R_{\Gamma_D}(u_{1}, u_{2}) + \lambda_{\Gamma_N}^R R_{\Gamma_N}(u_{1}, u_{2}),
\end{aligned}
\end{equation}
where $\bm{\lambda}^R = (\lambda_{r_1}^R, \lambda_{r_2}^R, \lambda_{\Gamma}^R ,\lambda_{b}^R)\geq 0$ (element-wise inequality), and $R_{r_1}$, $R_{r_2}$, $R_{b}$, $R_{\Gamma_D}$, $R_{\Gamma_N}$ are regularization functionals. Specifically,
\begin{equation*}
\begin{aligned}
&R_{r_1}(u_1) = \big[\mathcal{L}_1[u_1]\big]^2_{\Omega_1}, \quad R_{r_2}(u_2)=\big[\mathcal{L}_2[u_2]\big]^2_{\Omega_2} , \quad
R_{b}(u_2)=\big[\mathcal{B}[u_2]\big]_{\partial \Omega}^2,\\
&R_{\Gamma_D}(u_1,u_2)=\big[\mathcal{I}_D[u_1,u_2]\big]_{\Gamma}^2, \quad R_{\Gamma_N}(u_1,u_2)=\big[\mathcal{I}_N[u_1,u_2]\big]_{\Gamma}^2.
\end{aligned}
\end{equation*}
For the convenience of following analysis, we denote the expected loss of Eq. \eqref{eq:interface problem} (when $\bm{\lambda}^R=0$) by $\text{Loss}^{\text{PINN}}(u_1,u_2;\bm{\lambda})$. More precisely,
\begin{equation}\label{eq: the expected loss}
\begin{aligned}
\text{Loss}^{\text{PINN}}(u_1,u_2;\bm{\lambda}) =&
\lambda_{r_1} \norm{\mathcal{L}_1[u_1]- f_1}^2_{L^2(\Omega_1;\mu_{r_1})}
+\lambda_{r_2} \norm{\mathcal{L}_2[u_2]- f_2}^2_{L^2(\Omega_2;\mu_{r_2})}\\
&+\lambda_{b}\norm{\mathcal{B}[u_2] - \bm{g}}^2_{L^2(\partial\Omega;\mu_{b})}
+\lambda_{\Gamma_D}\norm{\mathcal{I}_D[u_1, u_2] - \bm{\varphi}}^2_{L^2(\Gamma;\mu_{\Gamma})}\\
&+\lambda_{\Gamma_N}\norm{\mathcal{I}_N[u_1, u_2]- \bm{\psi}}^2_{L^2(\Gamma;\mu_{\Gamma})}.
\end{aligned}
\end{equation}

\begin{remark}
It is noted that the use of the gradient-enhanced strategy and Lipschitz regularization is for the convergence analysis. We will numerically verify that such technologies do not affect performance.
\end{remark}

\begin{remark}
The present paper only considers the high regularity setting and that point-wise evaluations are well-defined. Specifically, it is required that $\mathcal{L}_i[u_i]\in C(\Omega_i)$ for $i=1,2$, $\mathcal{I}_D[u_1, u_2]\in C(\Gamma)$, $\mathcal{I}_N[u_1, u_2]\in C(\Gamma)$, and $\mathcal{B}[u_2]\in C(\partial\Omega)$ for all $(u_1,u_2)\in(\mathcal{H}_{1,\bm{m}},\mathcal{H}_{2,\bm{m}})$, and $f_i\in C(\Omega_i)$ for $i=1,2$, $\bm{\psi}\in C(\Gamma)$, $\bm{\varphi}\in C(\Gamma)$ and $\bm{g}\in C(\partial\Omega)$.
\end{remark}

\section{Main results\label{main results}}

We first present assumptions on the training data distributions based on the probability space filling arguments \cite{calder2019consistency} to guarantee that random samples drawn from probability distributions can fill up both the interior of the domains $\Omega_1$ and $\Omega_2$ as well as the boundary $\partial \Omega$ and interface $\Gamma$.
\begin{assumption}[Random sampling] \label{assumption:data-dist}
	For the interface problem \eqref{eq:interface problem}, let $\mu_{r_1}$, $\mu_{r_2}$, $\mu_{\Gamma}$ and $\mu_{b}$ be probability distributions defined on $\Omega_1$, $\Omega_2$, $\Gamma$ and $\partial \Omega$, respectively.
	Let $\rho_{r_1}(\rho_{r_2})$ be the probability density of $\mu_{r_1}(\mu_{r_2})$ with respect to $d$-dimensional Lebesgue measure on $\Omega_1(\Omega_2)$.
	Let $\rho_{\Gamma}(\rho_{b})$ be the probability density of $\mu_{\Gamma}(\mu_{b})$
	with respect to the $(d-1)$-dimensional Hausdorff measure on $\Gamma(\partial\Omega)$.
	\begin{enumerate}
		\item $\rho_{r_1}$, $\rho_{r_2}$, $\rho_{\Gamma}$ and $\rho_{b}$ are supported on $\overline{\Omega}_1$, $\overline{\Omega}_2$, $\Gamma$ and $\partial\Omega$, respectively.
		Also, $\inf_{\Omega_1} \rho_{r_1} > 0$, $\inf_{\Omega_2} \rho_{r_2} > 0$, $\inf_{\Gamma} \rho_{\Gamma} > 0$, and $\inf_{\partial\Omega} \rho_{b} > 0$.
		\item
		For $\epsilon > 0$, there exists partitions of $\Omega_1$, $\Omega_2$, $\Gamma$ and $\partial \Omega$, $\{\Omega_{1,j}^\epsilon\}_{j=1}^{K_{r_1}}$, $\{\Omega_{2,j}^\epsilon\}_{j=1}^{K_{r_2}}$, $\{\Gamma_j^\epsilon\}_{j=1}^{K_{\Gamma}}$
		and $\{\partial\Omega_j^\epsilon\}_{j=1}^{K_{b}}$
		that depend on $\epsilon$
		such that
		for each $j$,
		there are cubes $H_{\epsilon}(\mathbf{z}_{j}^{r_1})$, $H_{\epsilon}(\mathbf{z}_{j}^{r_2})$, $H_{\epsilon}(\mathbf{z}_{j}^\Gamma)$ and
		$H_{\epsilon}(\mathbf{z}_{j}^b)$ of side length $\epsilon$
		centered at $\mathbf{z}_{j}^{r_1} \in \Omega_{1,j}^\epsilon$, $\mathbf{z}_{j}^{r_2} \in \Omega_{2,j}^\epsilon$, $\mathbf{z}_{j}^\Gamma\in \Gamma_j^\epsilon$
		and $\mathbf{z}_{j}^b \in \partial\Omega_j^\epsilon$, respectively,
		satisfying $\Omega_{1,j}^\epsilon \subset H_{\epsilon}(\mathbf{z}_{j}^{r_1})$, $\Omega_{2,j}^\epsilon \subset H_{\epsilon}(\mathbf{z}_{j}^{r_2})$, $\Gamma_j^\epsilon \subset H_{\epsilon}(\mathbf{z}_{j}^\Gamma)$
		and $\partial \Omega_j^\epsilon \subset H_{\epsilon}(\mathbf{z}_{j}^b)$.
		\item
		There exists positive constants $c_{r_1}, c_{r_2}, c_{\Gamma}, c_{b}$ such that
		$\forall \epsilon > 0$,
		the partitions from the above satisfy
		$c_{r_1} \epsilon^{d} \le \mu_{r_1}(\Omega_{1,j}^\epsilon)$, $c_{r_2} \epsilon^{d} \le \mu_{r_2}(\Omega_{2,j}^\epsilon)$, $c_{\Gamma} \epsilon^{d-1} \le
		\mu_{\Gamma}(\Gamma_j^\epsilon)$
		and $c_{b} \epsilon^{d-1} \le
		\mu_{b}(\partial\Omega_j^\epsilon)$ for all $j$.
		
		There exists positive constants $C_{r_1}, C_{r_2}, C_{\Gamma}, C_{b}$ such that for
		$\forall \mathbf{x}_{r_1}\in \Omega_1$, $\mathbf{x}_{r_2}\in \Omega_2$, $\mathbf{x}_\Gamma\in \Gamma$ and $\mathbf{x}_b\in \partial\Omega$, we have
		$\mu_{r_1}(B_{\epsilon}(\mathbf{x}_{r_1}) \cap \Omega_1) \le C_{r_1}\epsilon^d$, $\mu_{r_2}(B_{\epsilon}(\mathbf{x}_{r_2}) \cap \Omega_2) \le C_{r_2}\epsilon^d$, $\mu_{\Gamma}(B_\epsilon(\mathbf{x}_{\Gamma}) \cap \Gamma) \le C_{\Gamma} \epsilon^{d-1}$
		and
		$\mu_{b}(B_\epsilon(\mathbf{x}_{b}) \cap \partial\Omega) \le C_{b} \epsilon^{d-1}$
		where
		$B_\epsilon(x)$ is a closed ball of radius $\epsilon$ centered at $x$.
	
		Here $C_{r_1}, c_{r_1}$ depend only on $(\Omega_1,\mu_{r_1})$, $C_{r_2}, c_{r_2}$ depend only on $(\Omega_2,\mu_{r_2})$,   $C_{\Gamma}, c_{\Gamma}$ depend only on $(\Gamma, \mu_{\Gamma})$ and $C_{b}, c_{b}$ depend only on $(\partial\Omega, \mu_{b})$.
	\end{enumerate}
\end{assumption}
In contrast to the traditional applications of deep learning, such as image classification and natural language processing, where the data distributions are unknown and data sampling is very expensive, the aforementioned assumptions are mild and easy to satisfy when solving interface problems, as the computation domain and interface are given and the data distribution is known (e.g., the uniform probability distribution).

In addition, for the loss function \eqref{def: GE LIP PINN empirical loss} to be well-defined, we have to make some assumptions about the interface problem~\eqref{eq:interface problem} and the hypothesis space of neural networks. Here, the network architecture $\bm{\overrightarrow{n}}$ is expected to grow proportionally to the number of training samples $\bm{m}$, thus we rewrite $\mathcal{H}_{\bm{\overrightarrow{n}}}^\text{NN}$ as $\mathcal{H}_{\bm{m}}$ for simplicity.
\begin{assumption}[Interface problem and hypothesis space] \label{assumption:convergence}
	Let $\mathcal{H}_{1,\bm{m}}$ and $\mathcal{H}_{2,\bm{m}}$ be the class of neural networks defined on $\overline{\Omega}_1$ and $\overline{\Omega}_2$, respectively.
	\begin{enumerate}
	    \item Let $f_1 \in C^{0,L}(\Omega_1)$, $f_2 \in C^{0,L}(\Omega_2)$, $\psi \in C^{1,L}(\Gamma)$, $\varphi\in C^{2,L}(\Gamma)$ and $g \in C^{2,L}(\partial\Omega)$.
	    \item
	    For each $\bm{m}$,  $\mathcal{H}_{1,\bm{m}}\subset C^{2,L}(\overline{\Omega}_1)$, $\mathcal{H}_{2,\bm{m}}\subset C^{2,L}(\overline{\Omega}_2)$
	    such that
    	for any  $(u_1,u_2) \in (\mathcal{H}_{1,\bm{m}},\mathcal{H}_{2,\bm{m}})$, $\mathcal{L}_1[u_1] \in C^{0,L}(\Omega_1)$, $\mathcal{L}_2[u_2] \in C^{0,L}(\Omega_2)$, $\mathcal{I}_{\Gamma_D}[u_{1}, u_{2}]\in C^{0,L}(\Gamma)$, $\mathcal{I}_{\Gamma_N}[u_{1}, u_{2}]\in C^{0,L}(\Gamma)$
	    and $\mathcal{B}[u_2] \in C^{0,L}(\Gamma)$.
	    \item For each $\bm{m}$, $\mathcal{H}_{1,\bm{m}}$($\mathcal{H}_{2,\bm{m}}$) contains a network $\hat{u}_{1,\bm{m}}$($\hat{u}_{2,\bm{m}}$)
		satisfying
		\begin{equation*}
		\text{Loss}_{\bm{m}}^{\text{PINN}}(\hat{u}_{1,\bm{m}}, \hat{u}_{2,\bm{m}};\bm{\lambda}) =  \mathcal{O}(\max\{ m_{r_1}, m_{r_2}, m_{\Gamma}^{\frac{d}{d-1}}, m_b^{\frac{d}{d-1}}\}^{-\frac{1}{2}-\frac{1}{d}}),
		\end{equation*}
			\item and
		\begin{equation*}
		\begin{aligned}
		&\sup_{\bm{m}} \big[\mathcal{L}_1[\hat{u}_{1,\bm{m}}]\big]_{\Omega_1}<\infty,
		\quad
		\sup_{\bm{m}} \big[\mathcal{L}_2[\hat{u}_{2,\bm{m}}]\big]_{\Omega_2}<\infty,
		\quad
		\sup_{\bm{m}} \big[\mathcal{B}[\hat{u}_{2,\bm{m}}]\big]_{\partial\Omega} < \infty,
		\\
		&\sup_{\bm{m}} \big[\mathcal{I}_{\Gamma_D}[\hat{u}_{1,\bm{m}}, \hat{u}_{2,\bm{m}}]\big]_{\Gamma}<\infty,
		\quad
		\sup_{\bm{m}} \big[\mathcal{I}_{\Gamma_N}[\hat{u}_{1,\bm{m}}, \hat{u}_{2,\bm{m}}]\big]_{\Gamma}<\infty.
		\end{aligned}
		\end{equation*}
	\end{enumerate}
\end{assumption}
Popular activation functions, such as sigmoid $\text{Sig}(z)$ and $\text{tanh}(z)$, could satisfy the Lipschitz condition. It is known that FNNs can simultaneously and uniformly approximate a continuous function and various of its partial derivatives \cite{pinkus1999approximation,cybenko1989approximation,hornik1990universal,guhring2021approximation,de2021approximation,shen2021neural}. In particular, standard multi-layer FNNs with a tanh activation function are capable of approximating arbitrary functions from the Sobolev space, provided sufficiently many hidden units are available \cite{guhring2021approximation}.
Thus, the third term in Assumption~\ref{assumption:convergence} can be attained.

With these assumptions, the main result is presented as follows.
\begin{theorem} \label{thm:main-elliptic}
Suppose Assumptions \ref{assumption:data-dist} and  \ref{assumption:convergence}
hold. Let $m_{r_1}$, $m_{r_2}$, $m_b$ and $m_{\Gamma}$ be the number of iid samples from $\mu_{r_1}$, $\mu_{r_2}$, $\mu_{b}$ and $\mu_{\Gamma}$, respectively, and $m_{r_2} = \mathcal{O}(m_{r_1})$, $m_{\Gamma} = \mathcal{O}(m_{r_1}^{\frac{d-1}{d}})$, $m_{b} = \mathcal{O}(m_{r_1}^{\frac{d-1}{d}})$. Let
\begin{equation*}
C_{\bm{m}} = 3\max\{\kappa_{r_1} \sqrt{d}^{d} m_{r_1}^{\frac{1}{2}}, \kappa_{r_2} \sqrt{d}^{d} m_{r_2}^{\frac{1}{2}}, \kappa_{b} \sqrt{d}^{d-1} m_{b}^{\frac{1}{2}}, \kappa_{\Gamma} \sqrt{d}^{d-1} m_{\Gamma}^{\frac{1}{2}}\},
\end{equation*}
where $\kappa_{r_1} = \frac{C_{r_1}}{c_{r_1}}$, $\kappa_{r_2} = \frac{C_{r_2}}{c_{r_2}}$, $\kappa_{b} = \frac{C_{b}}{c_{b}}$, $\kappa_{\Gamma} = \frac{C_{\Gamma}}{c_{\Gamma}}$. Let $\bm{\hat{\lambda}}_{\bm{m}}^R = (\hat{\lambda}_{r_1,\bm{m}}^R,\hat{\lambda}_{r_2,\bm{m}}^R,\hat{\lambda}_{b,\bm{m}}^R,\hat{\lambda}_{\Gamma_D,\bm{m}}^R, \hat{\lambda}_{\Gamma_N,\bm{m}}^R)$
be a vector where
\begin{equation*}
\begin{aligned}
    &\hat{\lambda}_{r_1,\bm{m}}^R = \frac{3\lambda_{r_1} dc_{r_1}^{-\frac{2}{d}}}{C_{\bm{m}}}\cdot m_{r_1}^{-\frac{1}{d}},
    \quad
    \hat{\lambda}_{r_2,\bm{m}}^R = \frac{3\lambda_{r_2} dc_{r_2}^{-\frac{2}{d}}}{C_{\bm{m}}}\cdot m_{r_2}^{-\frac{1}{d}}, \quad
    \hat{\lambda}_{b,\bm{m}}^R = \frac{3\lambda_{b} d c_{b}^{-\frac{2}{d-1}} }{C_{\bm{m}}}\cdot m_{b}^{-\frac{1}{d-1}},\\
    &\hat{\lambda}_{\Gamma_D,\bm{m}}^R = \frac{3\lambda_{\Gamma_D} d c_{\Gamma}^{-\frac{2}{d-1}} }{C_{\bm{m}}}\cdot m_{\Gamma}^{-\frac{1}{d-1}},
    \quad \hat{\lambda}_{\Gamma_N,\bm{m}}^R = \frac{3\lambda_{\Gamma_N} d c_{\Gamma}^{-\frac{2}{d-1}} }{C_{\bm{m}}}\cdot m_{\Gamma}^{-\frac{1}{d-1}}.
    \end{aligned}
\end{equation*}
Let $\bm{\lambda}_{\bm{m}}^R$ be a vector satisfying
\begin{equation*}
\bm{\lambda}_{\bm{m}}^R \ge \bm{\hat{\lambda}}_{\bm{m}}^R, \qquad
\norm{\bm{\lambda}_{\bm{m}}^R}_\infty = \mathcal{O}(\norm{\bm{\hat{\lambda}}_{\bm{m}}^R}_\infty).
\end{equation*}
Let $(u_{1, \bm{m}}, u_{2, \bm{m}}) \in (\mathcal{H}_{1,\bm{m}},\mathcal{H}_{2,\bm{m}})$ be a minimizer of the Lipschitz regularized loss $\text{Loss}_{\bm{m}}(\cdot;\bm{\lambda}, \bm{\lambda}_{\bm{m}}^R)$ \eqref{def: GE LIP PINN empirical loss}. Then the following holds,

\begin{itemize}
\item The interface problem \eqref{eq:interface problem} has a unique solution $u^* \in H^2(\Omega_1)\cap H^2(\Omega_2)$.

\item With probability 1 over iid samples,
\begin{equation*}
\lim_{m_{r_1} \to \infty} u_{1, \bm{m}} = u^*\ \text{in } H^2(\Omega_1), \quad \lim_{m_{r_1} \to \infty} u_{2, \bm{m}} = u^*\ \text{in } H^2(\Omega_2).
\end{equation*}
\end{itemize}
\end{theorem}
Theorem \ref{thm:main-elliptic} shows that the minimizers of the Lipschitz regularized empirical losses \eqref{def: GE LIP PINN empirical loss} converge to the unique solution to the interface problem \eqref{eq:interface problem} in $H^2$ as the number of samples increases. The proof is postponed to Section \ref{proofs}.

\begin{remark}
That $(u_{1, \bm{m}}, u_{2, \bm{m}}) \in (\mathcal{H}_{1,\bm{m}},\mathcal{H}_{2,\bm{m}})$ is a minimizer means
\begin{equation*}
\text{Loss}_{\bm{m}}(u_{1,\bm{m}},u_{2,\bm{m}};\bm{\lambda}, \bm{\lambda}_{\bm{m}}^R)\\
\leq \text{Loss}_{\bm{m}}(u_{1},u_{2};\bm{\lambda}, \bm{\lambda}_{\bm{m}}^R), \text{ for } \forall (u_{1}, u_{2}) \in (\mathcal{H}_{1,\bm{m}},\mathcal{H}_{2,\bm{m}}).
\end{equation*}
We remark that this condition can be relaxed to
\begin{equation*}
\text{Loss}_{\bm{m}}(u_{1,\bm{m}},u_{2,\bm{m}};\bm{\lambda}, \bm{\lambda}_{\bm{m}}^R)\\
\leq \text{Loss}_{\bm{m}}(\hat{u}_{1,\bm{m}},\hat{u}_{2,\bm{m}};\bm{\lambda}, \bm{\lambda}_{\bm{m}}^R), \text{ for $\hat{u}_{1,\bm{m}}$, $\hat{u}_{2,\bm{m}}$ given in Assumption \ref{assumption:convergence}.}
\end{equation*}
\end{remark}

\section{Numerical experiments\label{numerical examples}}

In this section, we present numerical evidence to verify our analysis. We limit ourselves to the idealized setting considered for the theoretical analysis and to two-dimensional (2D) interface problems for the sake of illustration. For simplicity, we refer to the results obtained by minimizing the empirical loss $\text{Loss}^{\text{PINN}}_{\bm{m}}(u_1,u_2;\bm{\lambda}=\bm{1})$ \eqref{def: vanilla PINN empirical loss} without gradient enhancement on boundary and interface, i.e., the original PINN loss \cite{raissi2019physics}, as ``PINN", to the results obtained by minimizing the gradient-enhanced PINN empirical loss \eqref{def: vanilla PINN empirical loss} as ``PINN-GE", and to the results obtained by minimizing the loss \eqref{def: GE LIP PINN empirical loss} as ``LIPR-GE". Throughout all benchmarks, we show the $L^2$ and $H^2$ convergence of the trained neural networks obtained by LIPR-GE as the number of training data increases. Note that the values of  $b_1=b_2=0$ (in Eq. \eqref{eq:interface problem}) are used for all the test examples except Example \ref{example2}, where the non-zero values are mentioned. All code and data accompanying this manuscript are publicly available at \url{https://github.com/bzlu-Group/ConvergencePINNInterface}.

\subsection{Settings}
\noindent\textbf{Network architecture.}
The feed-forward tanh-neural networks of depth 5 and width 200 are employed for all experiments.

\noindent\textbf{Training data.} The training points are randomly drawn from the corresponding domains. Specially, taking $m_r=10,30,50,100,300,\cdots,10000$, we randomly sample the training data points $\{\mathbf{x}_r^i\}_{i=1}^{m_r}$ inside the domain $\Omega$ and then divided them into two parts, i.e., $\{\mathbf{x}_{r_1}^i\}_{i=1}^{m_{r_1}}$ and $\{\mathbf{x}_{r_2}^i\}_{i=1}^{m_{r_2}}$, according to the interface. In addition, taking $m_b=m_\Gamma=10\lfloor m_r^{1/2}\rfloor$, we randomly sample the training data points $\{\mathbf{x}_b^i\}_{i=1}^{m_b}$ and $\{\mathbf{x}_\Gamma^i\}_{i=1}^{m_\Gamma}$ on the boundary and interface, respectively, see Fig.~\ref{fig: train test dat point} (left) for an illustration. Note that this strategy satisfies the conditions stated in Theorem \ref{thm:main-elliptic}, i.e., $m_{r_2} = \mathcal{O}(m_{r_1})$, $m_{\Gamma} = \mathcal{O}(m_{r_1}^{\frac{d-1}{d}})$, $m_{b} = \mathcal{O}(m_{r_1}^{\frac{d-1}{d}})$.

\noindent\textbf{Gradient enhancement.}
The boundary and interface are parameterized by $\vartheta$, i.e., $\mathbf{x}=(x_1(\vartheta),x_2(\vartheta))$, and the gradients of the functions defined on the boundary or interface are derived by the related parameterized functions.

\noindent\textbf{Optimization.} We train the networks for 10,000 stochastic gradient descent steps by minimizing the loss using the Adam optimizer \cite{Diederik2015Adam}. The initial learning rate is $1\times10^{-3}$, halved every 1000 iterations; and full-batch training is employed.

\noindent\textbf{Regularization.} For the Lipschitz regularized terms in loss function (\ref{def: GE LIP PINN empirical loss}), we use the maximum of the sup norm of the derivative over the set of training data points to estimate the Lipschitz constant, more precisely,
\begin{equation*}
\begin{aligned}
&\big[\mathcal{L}_1[u_1]\big]^2_{\Omega_1}=\max_{1\leq j\leq m_{r_1}}\norm{\nabla \mathcal{L}_1[u_1](\mathbf{x}_{r_1}^j)}_\infty^2,
\
\big[\mathcal{L}_2[u_2]\big]^2_{\Omega_2}=\max_{1\leq j\leq m_{r_2}}\norm{\nabla \mathcal{L}_2[u_2](\mathbf{x}_{r_2}^j)}_\infty^2,
\\
&\big[\mathcal{B}[u_2]\big]_{\partial \Omega}^2=\max_{1\leq j\leq m_b}\norm{\nabla \mathcal{B}[u_2](\mathbf{x}_b^j)}_\infty^2,\
\big[\mathcal{I}_D[u_1,u_2]\big]_{\Gamma}^2=\max_{1\leq j\leq m_\Gamma}\norm{\nabla \mathcal{I}_D[u_1,u_2](\mathbf{x}_\Gamma^j)}_\infty^2,
\\
&\big[\mathcal{I}_N[u_1,u_2]\big]_{\Gamma}^2=\max_{1\leq j\leq m_\Gamma}\norm{\nabla \mathcal{I}_N[u_1,u_2](\mathbf{x}_\Gamma^j)}_\infty^2.
\end{aligned}
\end{equation*}
The weights in loss function \eqref{def: GE LIP PINN empirical loss} are set as $\bm{\lambda} = (\lambda_{r_1}, \lambda_{r_2}, \lambda_{b}, \lambda_{\Gamma_D}, \lambda_{\Gamma_N})=\bm{1}$, $\lambda_{r_1}^R= \lambda_{r_2}^R=\frac{1}{m_r}$, and $\lambda_{\Gamma_D}^R=\lambda_{\Gamma_N}^R=\frac{1}{m_\Gamma \sqrt{m_r}}$, $ \lambda_{b}^R=\frac{1}{m_b \sqrt{m_r}}$, which satisfy the conditions stated in Theorem~\ref{thm:main-elliptic}.

\begin{figure}[htbp] 
	\centering
	\includegraphics[width=1\textwidth]{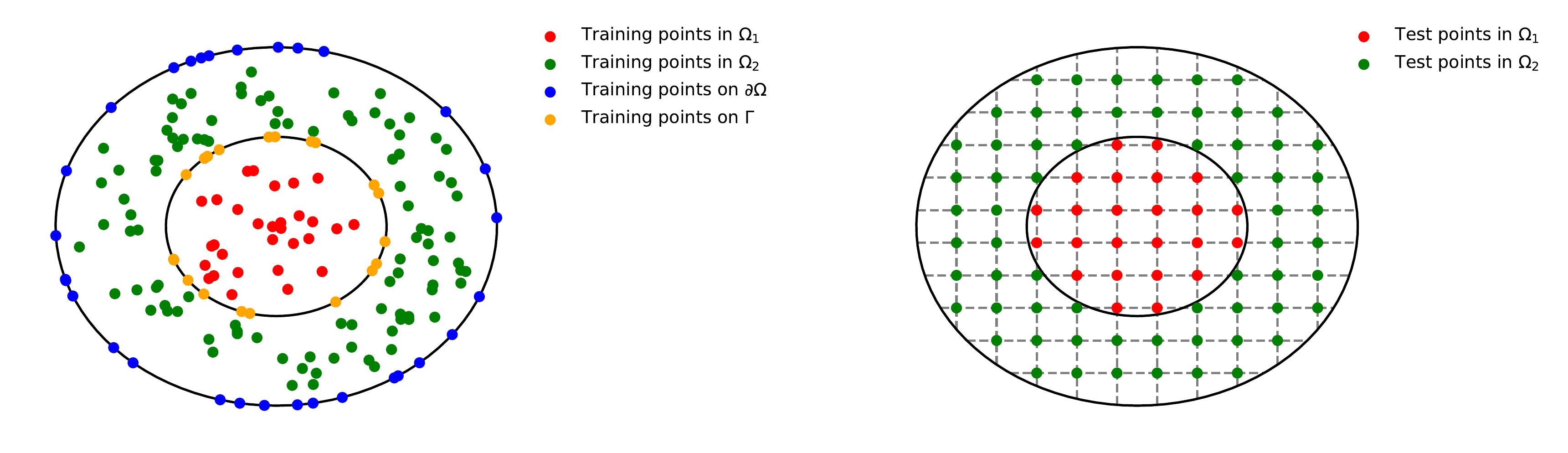}
	\caption{An illustration of the train and test data points. Left: These dots represent the training data points, which are randomly sampled in related regions. Right: An illustration of the equidistant test data points in the computational domain.}
	\label{fig: train test dat point}
\end{figure}

\noindent\textbf{Test.} After training, the $L^2$ error between the reference solution $u^*$ and the obtained neural network solution $\hat{u}$ is measured as
\begin{equation*}
\varepsilon_{L^2}=\norm{u-\hat{u}}_{L^2}:=\sqrt{\frac{1}{N}\sum_{i=1}^N \left|u^*(\mathbf{x}_i)-\hat{u}(\mathbf{x}_i)\right|^2},
\end{equation*}
where $N$ denotes the total number of test points in the computational domain, see Fig.~\ref{fig: train test dat point} (right) for an illustration. The $H^2$ error between $u^*$ and $\hat{u}$ is measured as
\begin{equation*}
\varepsilon_{H^2}=\left\{\sum_{|\bm{\tau}|\leq 2}\norm{D^{\bm{\tau}}(u^*-\hat{u})}^2_{L^2}\right\}^{\frac{1}{2}}.
\end{equation*}

\subsection{An elliptic interface problem with constant coefficients}\label{example1}
In this case, we consider Eq. \eqref{eq:interface problem} with a circle interface, which is given as $(x_1(\vartheta),x_2(\vartheta))=(1+\cos(\vartheta),1+\sin(\vartheta))$, where $\vartheta\in [0,2\pi]$. The computational domain in this problem is a closed disk with a radius of two and centered at $(1,1)$. The discontinuous coefficient $a$ is given as
\begin{equation*}
a(x_1,x_2)=\left\{
\begin{aligned}
&1,  \  &\text{in} \ \Omega_1, \\
&2,  \  &\text{in} \ \Omega_2. \\
\end{aligned}
\right.
\end{equation*}
The exact solution to this equation is given by
\begin{equation*}
u^*(x_1,x_2)=\left\{
\begin{aligned}
&2\tanh(x_1+x_2), \ &\text{in} \ \Omega_1, \\
&\tanh(x_1+x_2), \ &\text{in} \ \Omega_2. \\
\end{aligned}
\right.
\end{equation*}
Note that this solution can be exactly represented by the neural network we employed in this case. The corresponding  source term is $f(x_1,x_2)=8\tanh(x_1+x_2)-8\tanh(x_1+x_2)^3$ and the boundary and jump conditions can be found by using the exact solution.
\begin{figure}[htbp]
	\centering
	\includegraphics[width=1\textwidth]{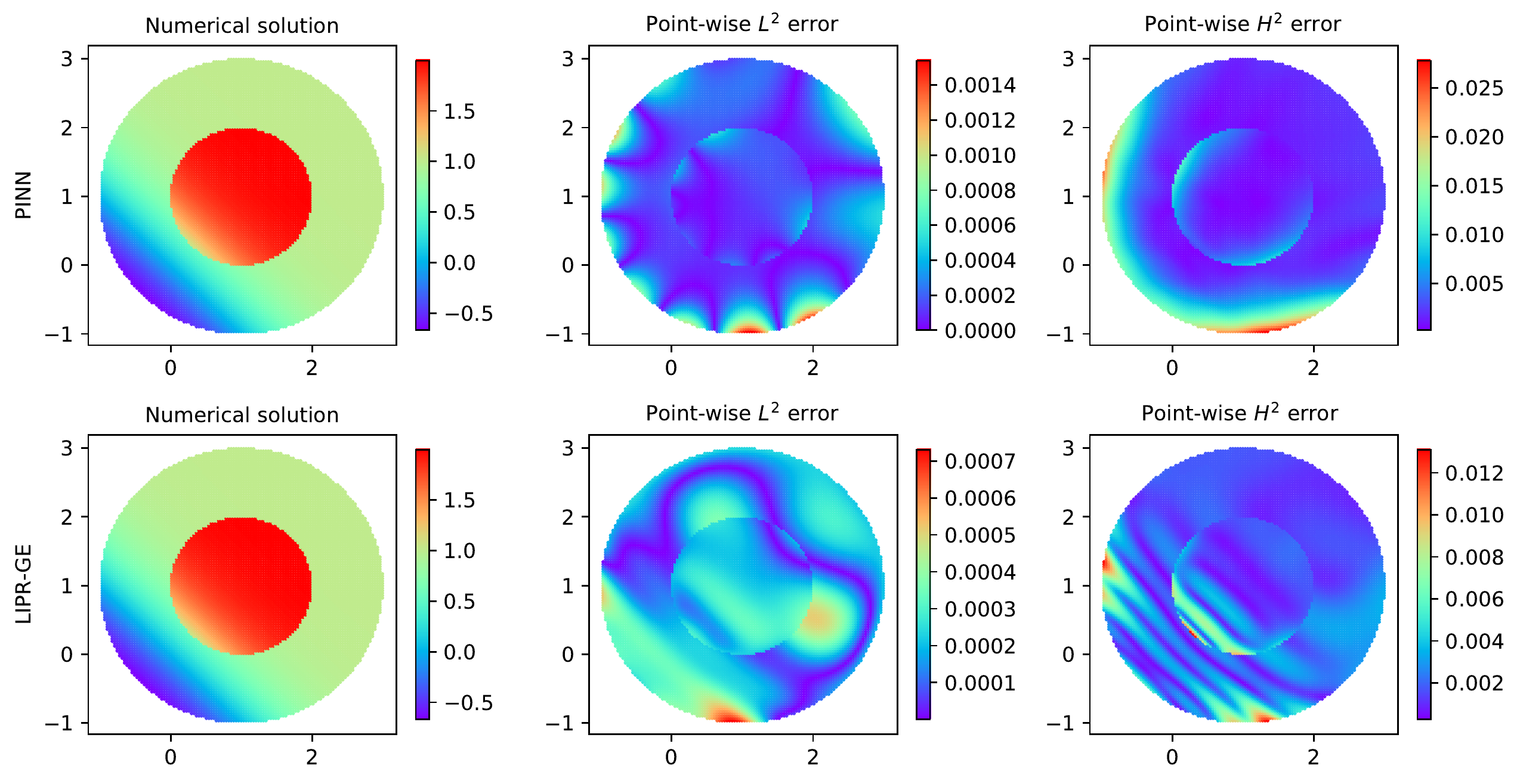}
	\caption{The numerical solution and point-wise errors for Example \ref{example1}. The first row gives the numerical results of PINN whereas the second row gives that of LIPR-GE. Here, the number of training data points $m_r=10000$.}
	\label{fig: point-wise error tanh}
\end{figure}
\begin{figure}[htbp]
	\centering
	\includegraphics[width=1\textwidth]{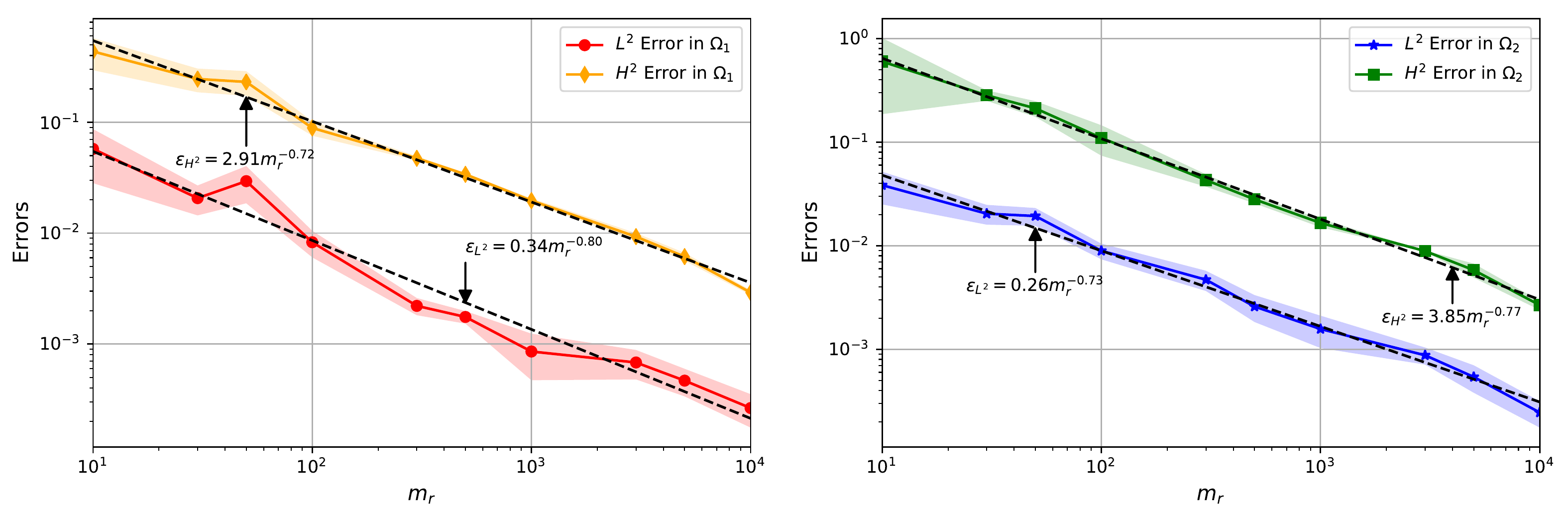}
	\caption{The $L^2$ and $H^2$ convergence of the errors of Example~\ref{example1} with respect to the number of training data points. The shaded regions are the one-standard-derivation from five runs with different training data and network initialization. Here, the number of test data points in $\Omega_1$ and $\Omega_2$ are $N_1=7833$ and $N_2=23584$, respectively. }
	\label{fig: error convergence example 1}
\end{figure}

In Fig. \ref{fig: point-wise error tanh}, we compare the numerical solution obtained by LIPR-GE with that obtained by PINN. It can be seen from this figure that both two numerical solutions have small $L^2$ and $H^2$ errors, while the result of LIPR-GE performs better than that of PINN. Numerical results indicate that PINN with gradient enhancement is acceptable in practice as it does not affect the performance of vanilla PINN.

In Fig. \ref{fig: error convergence example 1}, we show the $L^2$ and the $H^2$ errors of the results obtained by LIPR-GE with respect to the number of interior points $m_r$. Note that the number of points on the boundary and on the interface increases as $m_r$ increases. To show the convergence trend, we construct a univariate linear regression, i.e., $\log_{10}\varepsilon=\alpha\log_{10}m_r+\beta$, for the logarithm of the numerical solution error $\log_{10}\varepsilon$ (i.e., $\log_{10}\varepsilon_{L^2}$ or $\log_{10}\varepsilon_{H^2}$) versus $\log_{10}m_r$, and estimate the parameters $\alpha$ and $\beta$ using the linear least square algorithm. The dash lines here are the results of the regression. As expected by Theorem~\ref{thm:main-elliptic}, the $L^2$ and $H^2$ errors decrease as $m_r$ increases, implying the $L^2$- and $H^2$-convergence.

\subsection{An elliptic interface problem with high contrast coefficients}\label{example3}
Next, we consider Eq. \eqref{eq:interface problem} in the case of a large contrast in discontinuous coefficient $a$. Here, the computational domain $\Omega$ is a disk with a radius of one, centered at the origin. The interface is defined as $(x_1,x_2)=(r_0\cos(\vartheta),r_0\sin(\vartheta))$, where $r_0=0.5$. This  exact solution \cite{wang2020mesh} (in the polar coordinate) of this example is expressed as
\begin{equation*}
  u^*(r,\theta) = \left
\{\begin{aligned}
&\frac{r^3}{1000}, \ &r< r_0,\\
&r^3-\frac{999}{1000}r^3_0,\ &r\geq r_0,
\end{aligned}\right.
\end{equation*}
where $r=\sqrt{x_1^2+x_2^2}$ and the discontinuous coefficient is stated as
\begin{equation*}
a(x_1,x_2) = \left
\{\begin{aligned}
&1000,\ &&r< r_0,\\
&1,\ && r\geq r_0.
\end{aligned}\right.
\end{equation*}
Source terms, boundary and interface jump conditions are calculated from the above exact solution.

\begin{figure}[htbp]
	\centering
	\includegraphics[width=1\textwidth]{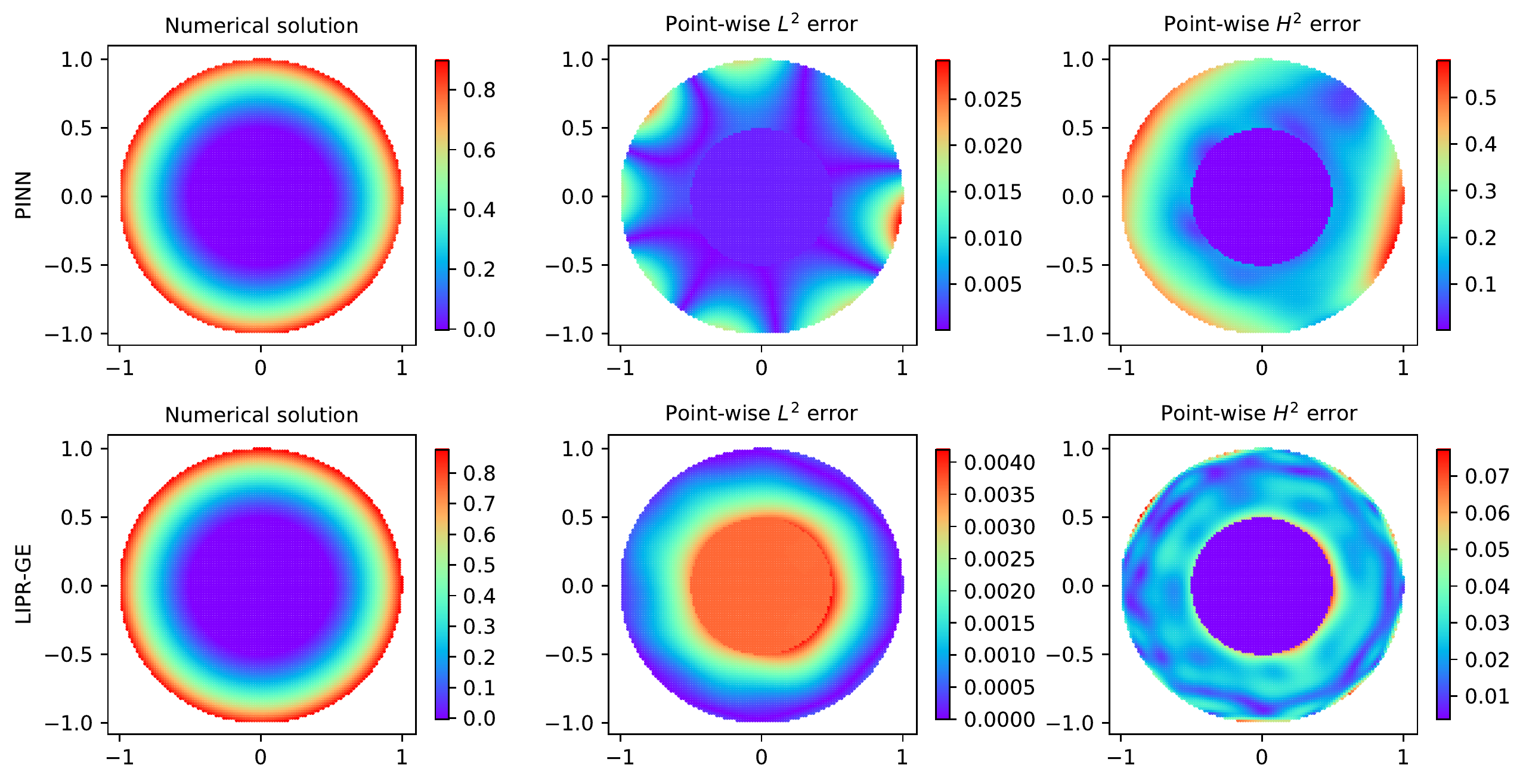}
	\caption{The numerical solution and point-wise errors for Example \ref{example3}. The first row gives the numerical results of PINN whereas the second row gives that of LIPR-GE. Here, the number of training data points $m_r=10000$.}
	\label{fig: compare high contrast}
\end{figure}
\begin{figure}[htbp]
	\centering
	\includegraphics[width=1\textwidth]{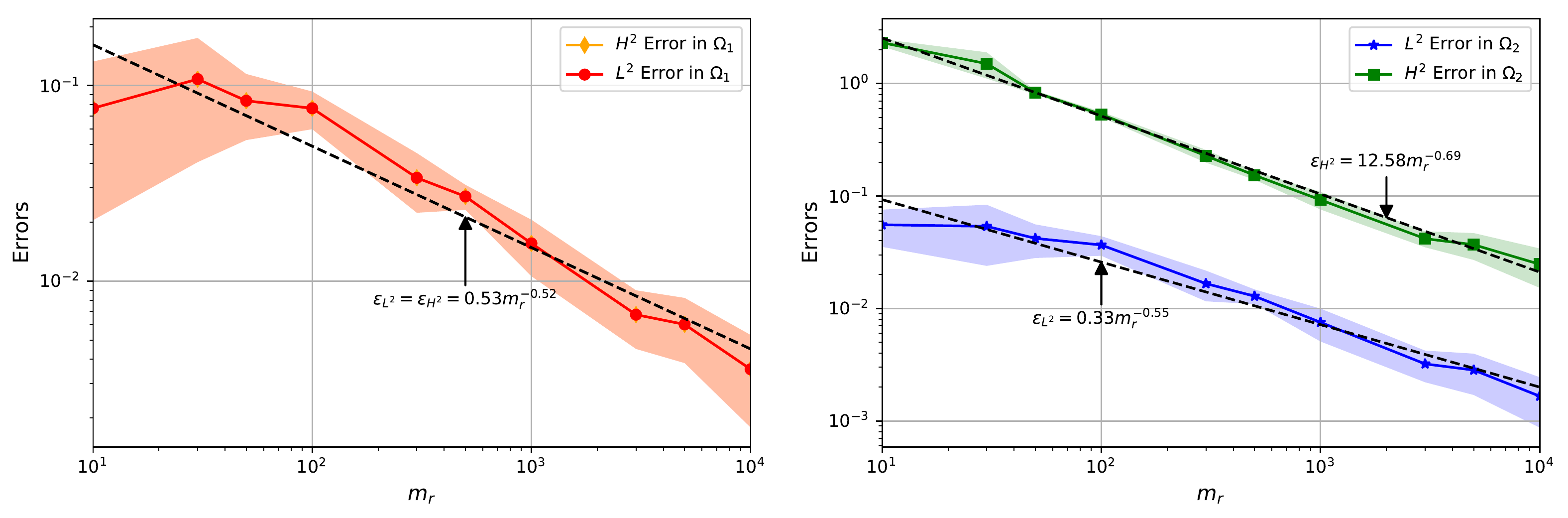}
	\caption{The $L^2$ and $H^2$ convergence of the errors of Example~\ref{example3} with respect to the number of training data points. The shaded regions are the one-standard-derivation from five runs with different training data and network initialization. Here, the number of test data points in $\Omega_1$ and $\Omega_2$ are $N_1=7829$ and $N_2=23588$, respectively.}
	\label{fig: high contrast}
\end{figure}

We first investigate the effect of gradient-enhanced strategies on the interfaces and depict the numerical results obtained by LIPR-GE and those obtained by PINN in Fig. \ref{fig: compare high contrast}. It can be observed that the auxiliary loss terms do not affect performance but also significantly reduce the absolute point-wise error at the interface and boundary. In addition, we continue testing the convergence. In Fig. \ref{fig: high contrast}, we show the $L^2$ and $H^2$ errors obtained by LIPR-GE with respect to the number of training data. We see that the rate of convergence is at least $\mathcal{O}(m_r^{-0.50})$. The results in this figure clearly demonstrate the convergence trend of $L^2$-error and $H^2$-error.

\subsection{An elliptic interface problem with variable coefficients}\label{example4}
In this case, we consider Eq. \eqref{eq:interface problem} with variable coefficients. Here, the computational domain is a closed disk placed at the origin with a radius of two, and the interface is circular with a radius of one and centered at the origin. The interface points can be obtained via $(x_1,x_2) = (\cos(\vartheta),\sin(\vartheta))$, where $\vartheta \in [0,2\pi)$. The coefficient $a$ is defined to be
\begin{equation*}
a(x_1,x_2)=\left\{
\begin{aligned}
&\cos(x_1+x_2)+2, \ &\text{in} \ \Omega_1, \\
&\sin(x_1+x_2)+2, \ &\text{in}  \ \Omega_2. \\
\end{aligned}
\right.
\end{equation*}
The exact solution to this problem is given by \cite{hou2005numerical}
\begin{equation*}
u^*(x_1,x_2)=\left\{
\begin{aligned}
&\sin(x_1+x_2), \ &\text{in} \ \Omega_1, \\
&\ln(x_1^2+x_2^2), \ &\text{in} \ \Omega_2, \\
\end{aligned}
\right.
\end{equation*}
and the corresponding source term is
\begin{equation*}
f(x_1,x_2)=\left\{
\begin{aligned}
&4\left(\cos(x_1+x_2)+1\right)\sin(x_1+x_2), \ &\text{in} \ \Omega_1, \\
&-2\cos(x_1+x_2)\frac{x_1+x_2}{x_1^2+x_2^2}, \ &\text{in} \ \Omega_2. \\
\end{aligned}
\right.
\end{equation*}

\begin{figure}[htbp] 
	\centering
	\includegraphics[width=1\textwidth]{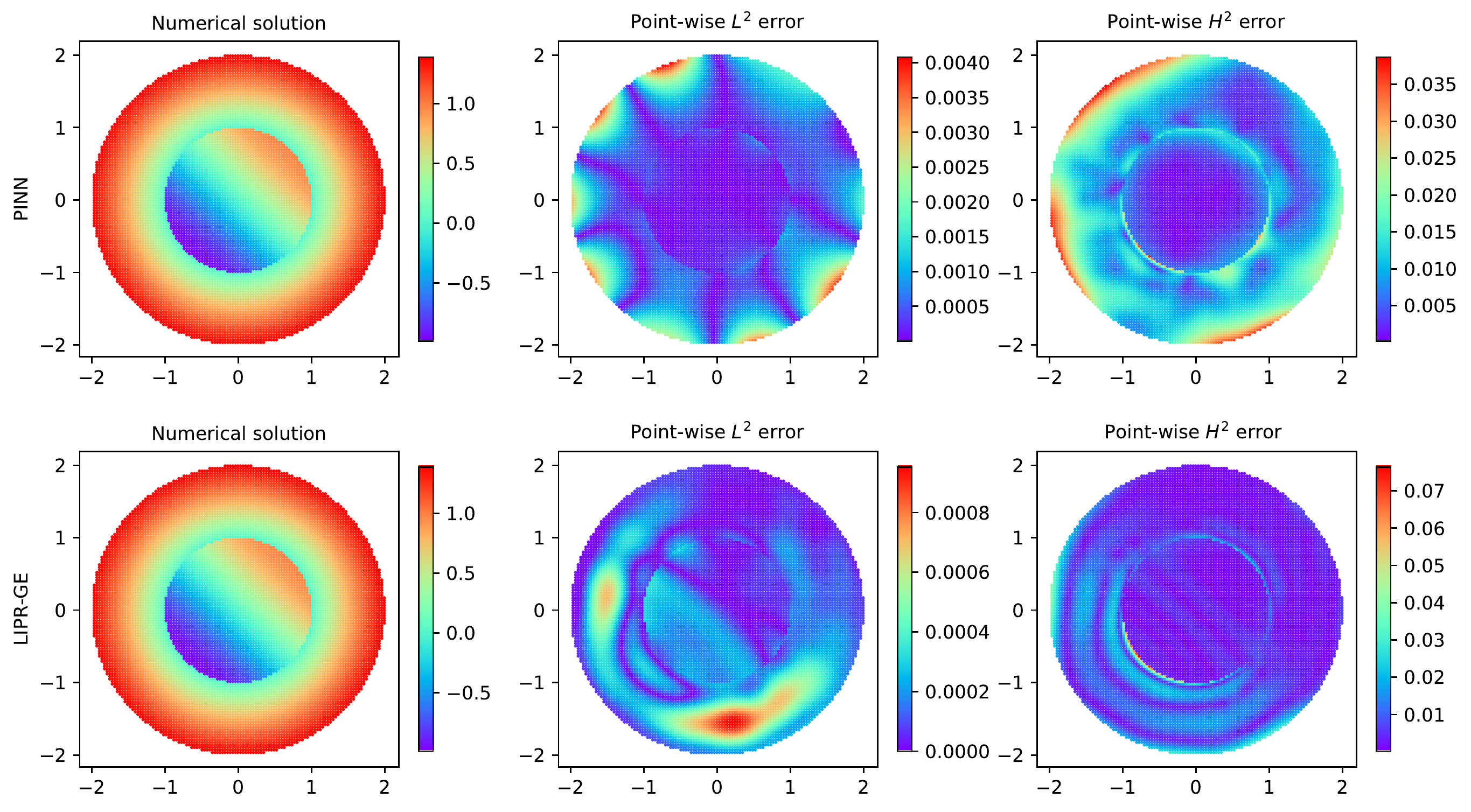}
	\caption{The numerical solution and point-wise errors for Example \ref{example4}. The first row gives the numerical results of PINN whereas the second row gives that of LIPR-GE. Here, the number of training data points $m_r=10000$.}
	\label{fig: point-wise error vc}
\end{figure}
\begin{figure}[htbp] 
	\centering
	\includegraphics[width=1\textwidth]{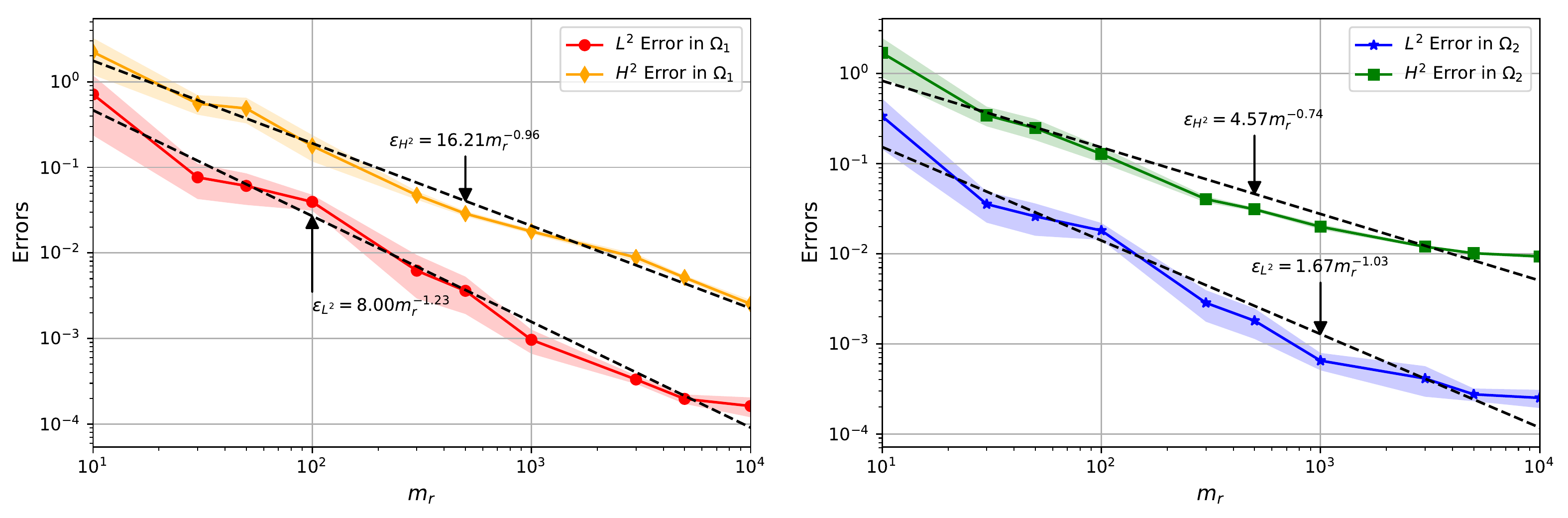}
	\caption{The $L^2$ and $H^2$ convergence of the errors of Example~\ref{example4} with respect to the number of training data points. The shaded regions are the one-standard-derivation from five runs with different training data and network initialization. Here, the number of test data points in $\Omega_1$ and $\Omega_2$ are $N_1=7829$ and $N_2=23556$, respectively.}
	\label{fig: error convergence example 4}
\end{figure}

Numerical results for Example \ref{example4} are displayed in Fig. \ref{fig: point-wise error vc} and Fig. \ref{fig: error convergence example 4}. In Fig. \ref{fig: point-wise error vc},  we first present a comparison between the exact and the numerical solution obtained using PINN or LIPR-GE. It can be observed that the solution obtained by LIPR-GE is in good agreement with that of PINN. As expected, the accuracy of LIPR-GE is not affected by the gradient enhancement and Lipschitz regularization. Furthermore, Fig. \ref{fig: error convergence example 4} shows the $L^2$- and $H^2$- convergence for neural network solutions in subdomains $\Omega_1$ (left) and $\Omega_2$ (right).
Again, we observe that the errors in both subdomains decreases rapidly as the number of training data points $m_r$ increases.
The theoretical results in Theorem \ref{thm:main-elliptic} are still valid in solving elliptic interface problems with variable coefficients.

\subsection{An elliptic interface problem with irregular geometry}\label{example2}
In this case, we consider Eq. \eqref{eq:interface problem} with a complicated interface $\Gamma$ (see Fig.~\ref{fig: example2 irregular domain} left), which consists of both convex and concave curves and is expressed with the following parametric equations
\begin{equation*}
\begin{aligned}
    x_1(\vartheta)=r\cos(\vartheta),\
    x_2(\vartheta)=r\sin(\vartheta),
\end{aligned}
\end{equation*}
where $r=1+0.36\sin(3\vartheta)+0.16\cos(2\vartheta)+0.4\cos(5\vartheta)$, $\vartheta\in[0,2\pi]$. Computational domain is shown in Fig.~\ref{fig: example2 irregular domain}. The boundary points (in polar coordinates) are obtained as $x_1=1.5r\cos(\vartheta)$ and $x_2=1.5r\sin(\vartheta)$, where
$r=1.5+0.14\sin(4\vartheta)+0.12\cos(6\vartheta)+0.09\cos(5\vartheta)$, $\vartheta \in [0,2\pi)$. The coefficient $b_1=b_2=-1$, and the discontinuous coefficient $a$ is defined to be
\begin{equation*}
a(x,y)=\left\{
\begin{aligned}
&x_1x_2, \ &\text{in} \ \Omega_1, \\
&x_1^2+x_2^2, \ &\text{in}  \ \Omega_2. \\
\end{aligned}
\right.
\end{equation*}
The exact solution is set to be
\begin{equation*}
u^*(x_1,x_2)=\left\{
\begin{aligned}
&\sin(x_1+x_2), \ &\text{in} \ \Omega_1, \\
&\cos(x_1+x_2), \ &\text{in} \ \Omega_2. \\
\end{aligned}
\right.
\end{equation*}
The necessary source terms, boundary and interface jump conditions can be derived from this exact solution.

\begin{figure}[htbp] 
	\centering
	\includegraphics[width=1\textwidth]{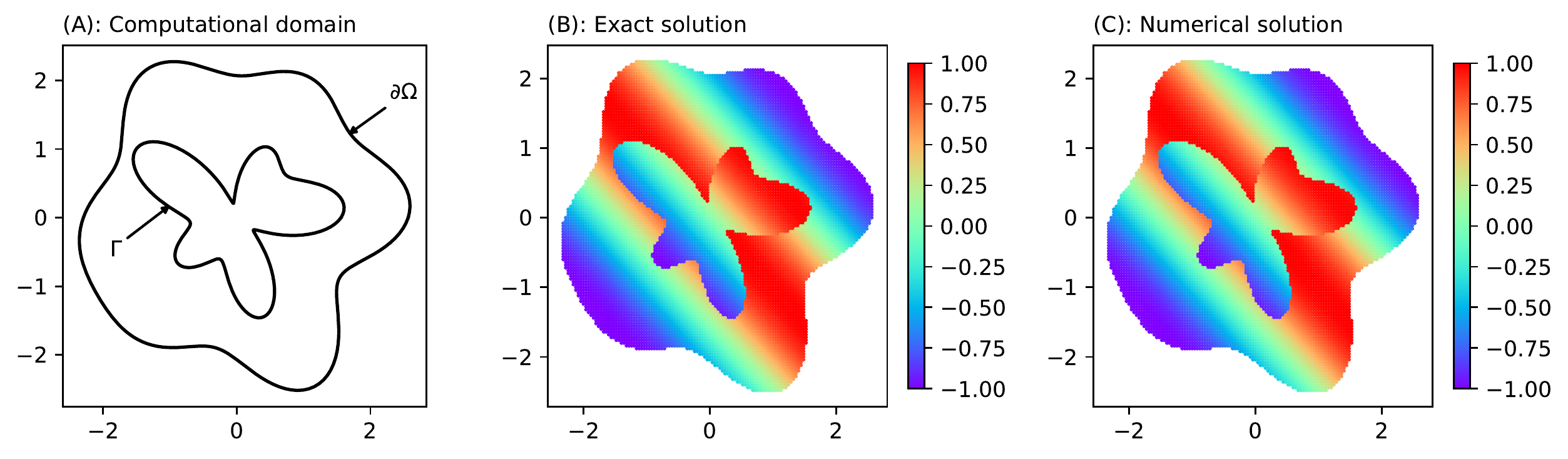}
	\caption{The computational domain, exact and numerical solutions for Example~\ref{example2}. Here, the number of training data points $m_r=10000$.}
	\label{fig: example2 irregular domain}
\end{figure}
\begin{figure}[htbp] 
	\centering
	\includegraphics[width=1\textwidth]{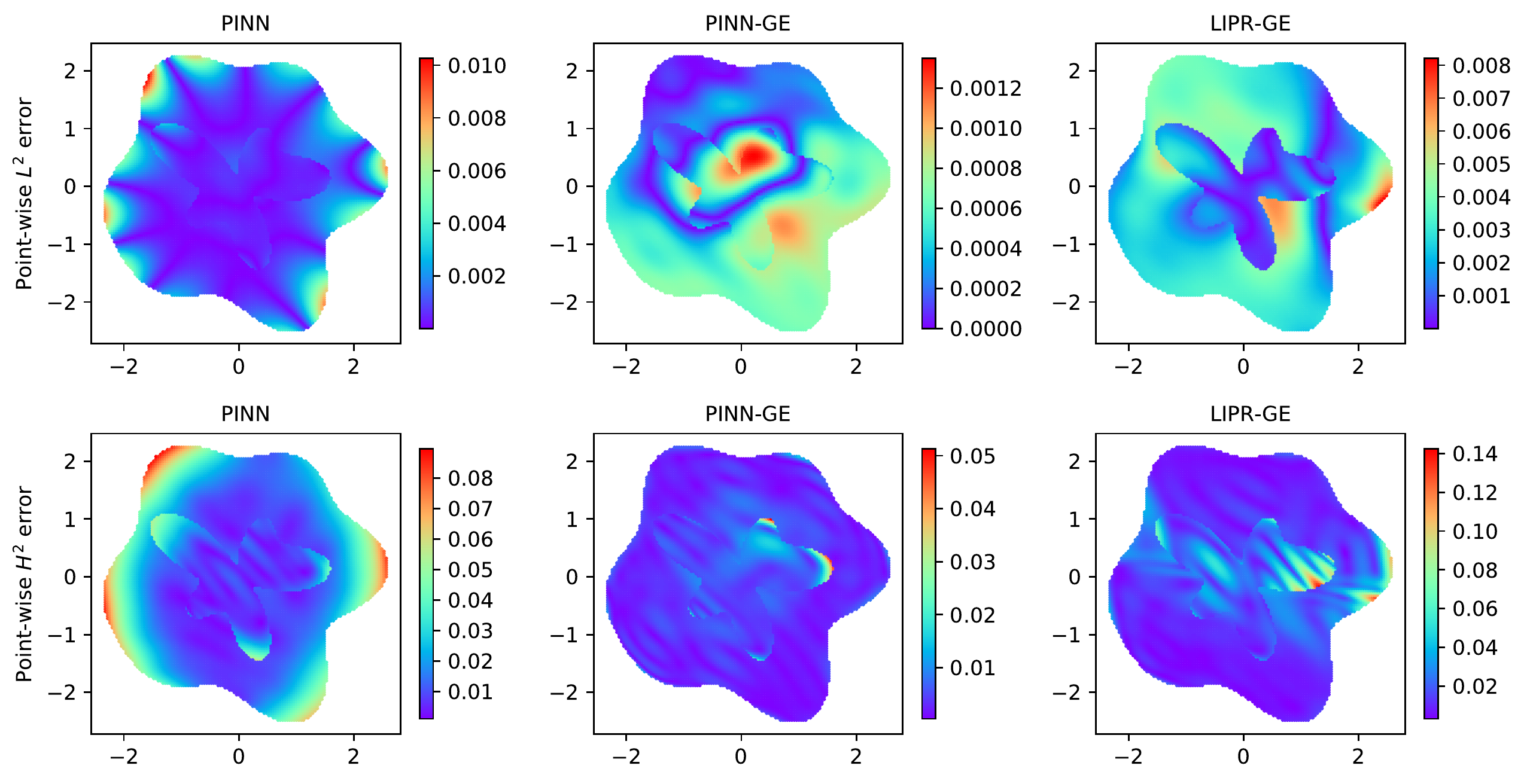}
	\caption{Point-wise errors for Example \ref{example2}. The errors of PINN, PINN-GE, and LIPR-GE are placed in the first, second, and third columns, respectively. Here, the number of training data points $m_r=10000$.}
	\label{fig: point-wise error cd}
\end{figure}
\begin{figure}[htbp] 
	\centering
	\includegraphics[width=1\textwidth]{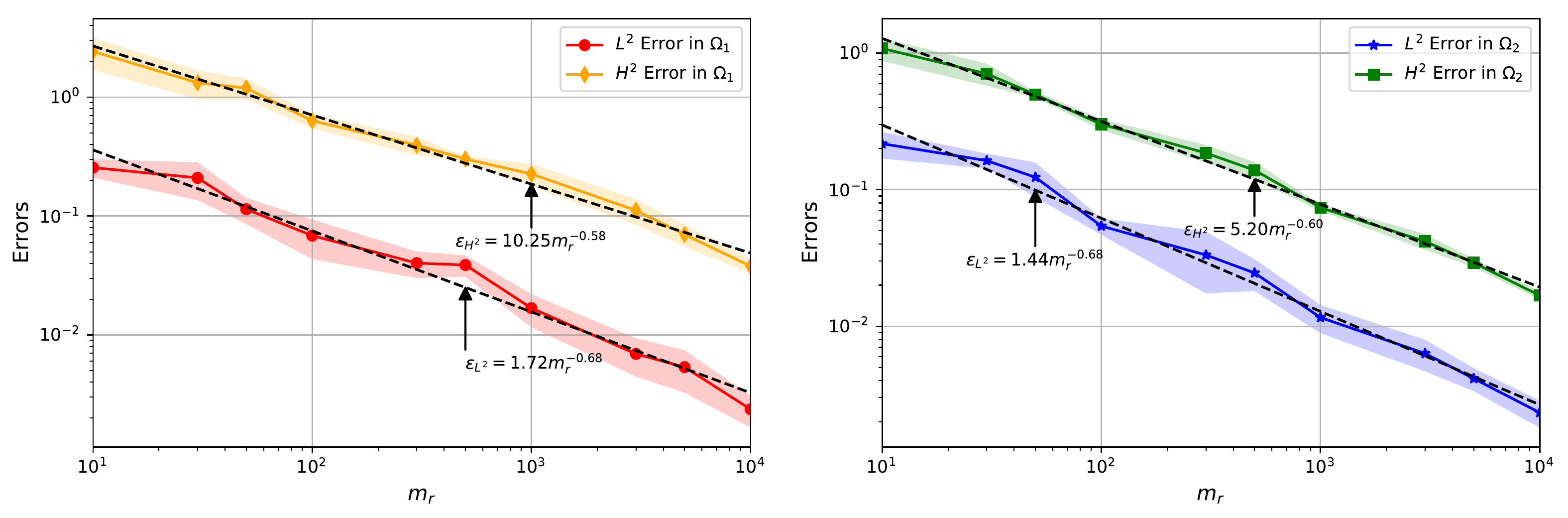}
	\caption{The $L^2$ and $H^2$ convergence of the errors of Example~\ref{example2} with respect to the number of training data points. The shaded regions are the one-standard-derivation from five runs with different training data and network initialization. Here, the number of test data points in $\Omega_1$ and $\Omega_2$ are $N_1=4044$ and $N_2=13787$, respectively.}
	\label{fig: error convergence example2}
\end{figure}

The computational domain, the exact solution of Example \eqref{example2}, and the numerical solution obtained by LIPR-GE are shown in Fig. \ref{fig: example2 irregular domain}. And the point-wise $L^2$ and $H^2$ errors in the whole domain for PINN, PINN-GE, and LIPR-GE are presented in Fig. \ref{fig: point-wise error cd}. It is observed that LIPR-GE is less accurate than  Vanilla PINN and PINN-GE. One explanation is that the additional auxiliary loss terms, especially the Lipschitz regularization, make it difficult for the optimization to find the minimizer. However, our following results clearly demonstrate that LIPR-GE can still recover the exact solution up to $O(10^{-3})$ accuracy in $L^2$ and $O(10^{-2})$ accuracy in $H^2$. Fig. \ref{fig: error convergence example2} summarizes the convergent evolution of the $L^2$ and $H^2$ errors obtained by LIPR-GE with respect to the number of training data points. Clearly, the numerical results demonstrate both the $L^2$ and $H^2$-convergence of the errors, which are consistent to the theoretical analysis of this paper.

\section{Proofs\label{proofs}}
We present the proof of Theorem \ref{thm:main-elliptic} in this section. The technique used in the following proof is similar to that used in the proof of Theorem 3 of \cite{shin2020on}. However, our Theorem \ref{thm:main-elliptic} applies to elliptic interface problems without a zero-loss assumption of interface and boundary conditions. This prevents direct use of the result from \cite{shin2020on}, which applies only to elliptic PDEs with a network solution obeying the boundary conditions exactly.
Throughout this section, we assume that Assumptions \ref{assumption:data-dist} and  \ref{assumption:convergence} hold. Among the crucial technical tools used here are some Sobolev inequality and probability space filling arguments \cite{calder2019consistency}. We start with the following auxiliary lemma:

\begin{lemma} \label{app:lemma-lip}
Suppose Assumption \ref{assumption:data-dist} holds. For training data $\mathcal{T}_{r_1}^{m_{r_1}} = \{\mathbf{x}_{r_1}^i\}_{i=1}^{m_{r_1}}$, $\mathcal{T}_{r_2}^{m_{r_2}} = \{\mathbf{x}_{r_2}^i\}_{i=1}^{m_{r_2}}$, $\mathcal{T}_{b}^{m_{b}} = \{\mathbf{x}_{b}^i\}_{i=1}^{m_{b}}$
and $\mathcal{T}_{\Gamma}^{m_{\Gamma}} = \{\mathbf{x}_{\Gamma}^i\}_{i=1}^{m_{\Gamma}}$, if $m_{r_1}$, $m_{r_2}$, $m_b$ and $m_{\Gamma}$ are large enough to satisfy that there exists $\mathbf{x}_{r_1}' \in \mathcal{T}_{r_1}$, $\mathbf{x}_{r_2}' \in \mathcal{T}_{r_2}$, $\mathbf{x}_{b}' \in \mathcal{T}_{b}$ and $\mathbf{x}_{\Gamma}' \in \mathcal{T}_{\Gamma}$ such that
$\norm{\mathbf{x}_{r_1}- \mathbf{x}_{r_1}'}_2 \le \epsilon_{r_1}$, $\norm{\mathbf{x}_{r_2}- \mathbf{x}_{r_2}'}_2 \le \epsilon_{r_2}$, $\norm{\mathbf{x}_b - \mathbf{x}_b' }_2 \le \epsilon_b$ and
$\norm{\mathbf{x}_{\Gamma} - \mathbf{x}_{\Gamma}'}_2 \le \epsilon_{\Gamma}$ for any $\mathbf{x}_{r_1} \in \Omega_1$, $\mathbf{x}_{r_2} \in \Omega_2$, $\mathbf{x}_{b} \in \partial\Omega$ and $\mathbf{x}_{\Gamma} \in \Gamma$,
then, we have
\begin{equation*}
\begin{aligned}
\text{Loss}^{\text{PINN}}(u_1,u_2;\bm{\lambda})
\le &
C_{\bm{m}} \cdot \text{Loss}_{\bm{m}}^{\text{PINN}}(u_1,u_2;\bm{\lambda})+3\lambda_{r_1}\epsilon_{r_1}^{2}\left(\big[\mathcal{L}_1[u_1]\big]_{\Omega_1}^2+\big[f_1\big]_{\Omega_1}^2\right)\\
&+3\lambda_{r_2}\epsilon_{r_2}^{2}\left(\big[\mathcal{L}_2[u_2]\big]_{\Omega_2}^2+\big[f_2\big]_{\Omega_2}^2\right)
+3\lambda_{b}\epsilon_{b}^{2}\left(\big[\mathcal{B}[u_2]\big]_{\partial \Omega}^2+\big[\bm{g}\big]_{\partial \Omega}^2\right)\\ &+3\lambda_{\Gamma_D}\epsilon_{\Gamma}^{2}\left(\big[\mathcal{I}_D[u_1,u_2]\big]_{\Gamma}^2
+\big[\bm{\varphi}\big]_{\Gamma}^2\right)
+3\lambda_{\Gamma_N}\epsilon_{\Gamma}^{2}\left(\big[\mathcal{I}_N[u_1,u_2]\big]_{\Gamma}^2
+\big[\bm{\psi}\big]_{\Gamma}^2\right),
\end{aligned}
\end{equation*}
where $C_{r_1}$, $C_{r_2}$, $C_{b}$, $C_{\Gamma}$ are those defined in Assumption \ref{assumption:data-dist}
and
$$
C_{\bm{m}} = 3\max\{C_{r_1}m_{r_1}\epsilon_{r_1}^{d}, C_{r_2}m_{r_2}\epsilon_{r_2}^{d},
C_{b}m_{b}\epsilon_{b}^{d-1}, C_{\Gamma}m_{\Gamma}\epsilon_{\Gamma}^{d-1} \}.
$$
\end{lemma}
\begin{proof}
As a consequence of Cauchy's inequality, i.e.,
$
\norm{\mathbf{x} + \mathbf{y} + \mathbf{z}}_2^2 \le 3(\norm{\mathbf{x}}_2^2 + \norm{\mathbf{y}}_2^2 + \norm{\mathbf{z}}_2^2)
$ for any three vectors $\mathbf{x}, \mathbf{y}, \mathbf{z}$,
we deduce that for $\mathbf{x}_{r_i}, \mathbf{x}_{r_i}' \in \Omega_i, i=1,2$,
\begin{equation*}
\begin{aligned}
\left|\mathcal{L}_i[u_i](\mathbf{x}_{r_i}) - f_i(\mathbf{x}_{r_i})\right|^2
\leq 3\left(\left|\mathcal{L}_i[u_i](\mathbf{x}_{r_i}) -\mathcal{L}_i[u_i](\mathbf{x}_{r_i}')\right|^2 + \left|\mathcal{L}_i[u_i](\mathbf{x}_{r_i}') -f_i(\mathbf{x}_{r_i}')\right|^2
+ \left|f_i(\mathbf{x}_{r_i}') - f_i(\mathbf{x}_{r_i})\right|^2 \right).
\end{aligned}
\end{equation*}
Similarly, for $\mathbf{x}_{b}, \mathbf{x}_{b}' \in \partial \Omega$, we have
\begin{equation*}
\norm{\mathcal{B}[u_2](\mathbf{x}_{b}) - \bm{g}(\mathbf{x}_{b})}_2^2
\leq 3\left(\norm{\mathcal{B}[u_2](\mathbf{x}_{b})- \mathcal{B}[u_2](\mathbf{x}_{b}')}_2^2
+\norm{\mathcal{B}[u_2](\mathbf{x}_{b}') - \bm{g}(\mathbf{x}_{b}')}_2^2
 +\norm{\bm{g}(\mathbf{x}_{b}')- \bm{g}(\mathbf{x}_{b})}_2^2\right),
\end{equation*}
and for $\mathbf{x}_{\Gamma}, \mathbf{x}_{\Gamma}' \in \Gamma$, we have
\begin{equation*}
\begin{aligned}
\norm{\mathcal{I}_D[u_1,u_2](\mathbf{x}_{\Gamma}) - \bm{\varphi}(\mathbf{x}_{\Gamma})}_2^2
\leq 3\big(&\norm{\mathcal{I}_D[u_1,u_2](\mathbf{x}_{\Gamma})- \mathcal{I}_D[u_1,u_2](\mathbf{x}_{\Gamma}')}_2^2+\norm{\mathcal{I}_D[u_1,u_2](\mathbf{x}_{\Gamma}') - \bm{\varphi}(\mathbf{x}_{\Gamma}')}_2^2\\ &+\norm{\bm{\varphi}(\mathbf{x}_{\Gamma})- \bm{\varphi}(\mathbf{x}_{\Gamma}')}_2^2\big),
\end{aligned}
\end{equation*}
\begin{equation*}
\begin{aligned}
\norm{\mathcal{I}_N[u_1,u_2](\mathbf{x}_{\Gamma}) - \bm{\varphi}(\mathbf{x}_{\Gamma})}_2^2
\leq 3\big(&\norm{\mathcal{I}_N[u_1,u_2](\mathbf{x}_{\Gamma})- \mathcal{I}_N[u_1,u_2](\mathbf{x}_{\Gamma}')}_2^2+\norm{\mathcal{I}_N[u_1,u_2](\mathbf{x}_{\Gamma}') - \bm{\psi}(\mathbf{x}_{\Gamma}')}_2^2\\
&+\norm{\bm{\psi}(\mathbf{x}_{\Gamma})- \bm{\psi}(\mathbf{x}_{\Gamma}')}_2^2\big).
\end{aligned}
\end{equation*}

In addition, by the conditions, for $\forall \mathbf{x}_{r_1} \in \Omega_1$, $\forall \mathbf{x}_{r_2} \in \Omega_2$, $\forall \mathbf{x}_{b} \in \partial\Omega$ and $\forall \mathbf{x}_{\Gamma} \in \Gamma$, there exist $\mathbf{x}_{r_1}' \in \mathcal{T}_{r_1}^{m_{r_1}}$, $\mathbf{x}_{r_2}' \in \mathcal{T}_{r_2}^{m_{r_2}}$, $\mathbf{x}_{b}' \in \mathcal{T}_{b}^{m_b}$ and $\mathbf{x}_{\Gamma}' \in \mathcal{T}_{\Gamma}^{m_\Gamma}$ such that $\norm{\mathbf{x}_{r_1}- \mathbf{x}_{r_1}'}_2 \le \epsilon_{r_1}$, $\norm{\mathbf{x}_{r_2}- \mathbf{x}_{r_2}'}_2 \le \epsilon_{r_2}$, $\norm{\mathbf{x}_b - \mathbf{x}_b'}_2 \le \epsilon_b$ and $\norm{\mathbf{x}_{\Gamma} - \mathbf{x}_{\Gamma}'}_2 \le \epsilon_{\Gamma}$.
Taking
\begin{equation*}
\begin{aligned}
\mathbf{L}(\mathbf{x}_{r_1}, \mathbf{x}_{r_2}, \mathbf{x}_b, \mathbf{x}_\Gamma;u_1,u_2,\bm{\lambda},\bm{0})
=&\lambda_{r_1}\left|\mathcal{L}_1[u_1](\mathbf{x}_{r_1}) - f_1(\mathbf{x}_{r_1})\right|^2
+\lambda_{r_2}\left|\mathcal{L}_2[u_2](\mathbf{x}_{r_2}) - f_2(\mathbf{x}_{r_2})\right|^2\\
&+\lambda_{b}\norm{\mathcal{B}[u_2](\mathbf{x}_{b}) - \bm{g}(\mathbf{x}_{b})}_2^2
+\lambda_{\Gamma_D}\norm{\mathcal{I}_D[u_1,u_2](\mathbf{x}_{\Gamma})-\bm{\varphi}(\mathbf{x}_{\Gamma})}_2^2\\
&+\lambda_{\Gamma_N}\norm{\mathcal{I}_N[u_1,u_2](\mathbf{x}_{\Gamma})-\bm{\psi}(\mathbf{x}_{\Gamma})}_2^2,
\end{aligned}
\end{equation*}
we have that
\begin{equation*}
\begin{aligned}
&\mathbf{L}(\mathbf{x}_{r_1}, \mathbf{x}_{r_2}, \mathbf{x}_b, \mathbf{x}_\Gamma;u_1,u_2,\bm{\lambda},\bm{0})\\
\leq& 3\mathbf{L}(\mathbf{x}_{r_1}',\mathbf{x}_{r_2}',\mathbf{x}_b',\mathbf{x}_\Gamma';u_1,u_2,\bm{\lambda},\bm{0})
+3\lambda_{r_1}\left(\left|\mathcal{L}_1[u_1](\mathbf{x}_{r_1}) -\mathcal{L}_1[u_1](\mathbf{x}_{r_1}')\right|^2+\left|f_1(\mathbf{x}_{r_1}) - f_1(\mathbf{x}_{r_1}')\right|^2\right)
\\
&+3\lambda_{r_2}\left(\left|\mathcal{L}_2[u_2](\mathbf{x}_{r_2}) -\mathcal{L}_2[u_2](\mathbf{x}_{r_2}')\right|^2+\left|f_2(\mathbf{x}_{r_2}) - f_2(\mathbf{x}_{r_2}')\right|^2\right)\\
&+3\lambda_b\left(\norm{\mathcal{B}[u_2](\mathbf{x}_{b})- \mathcal{B}[u_2](\mathbf{x}_{b}')}_2^2
+\norm{\bm{g}(\mathbf{x}_{b}')- \bm{g}(\mathbf{x}_{b})}_2^2\right)
\\
&+3\lambda_{\Gamma_D}\left(\norm{\mathcal{I}_D[u_1,u_2](\mathbf{x}_{\Gamma})- \mathcal{I}_D[u_1,u_2](\mathbf{x}_{\Gamma}')}_2^2
+\norm{\bm{\varphi}(\mathbf{x}_{\Gamma})-\bm{\varphi}(\mathbf{x}_{\Gamma}')}_2^2\right)
\\
&+3\lambda_{\Gamma_N}\left(\norm{\mathcal{I}_N[u_1,u_2](\mathbf{x}_{\Gamma})- \mathcal{I}_N[u_1,u_2](\mathbf{x}_{\Gamma}')}_2^2
+\norm{\bm{\psi}(\mathbf{x}_{\Gamma})-\bm{\psi}(\mathbf{x}_{\Gamma}')}_2^2\right)\\
\leq&  3\mathbf{L}(\mathbf{x}_{r_1}',\mathbf{x}_{r_2}',\mathbf{x}_b',\mathbf{x}_\Gamma';u_1,u_2,\bm{\lambda},\bm{0})
+3\lambda_{r_1}\epsilon_{r_1}^{2}\left(\big[\mathcal{L}_1[u_1]\big]_{\Omega_1}^2+\big[f_1\big]_{\Omega_1}^2\right)
+3\lambda_{r_2}\epsilon_{r_2}^{2}\left(\big[\mathcal{L}_2[u_2]\big]_{\Omega_2}^2+\big[f_2\big]_{\Omega_2}^2\right)\\
&+3\lambda_{b}\epsilon_{b}^{2}\left(\big[\mathcal{B}[u_2]\big]_{\partial \Omega}^2
+\big[\bm{g}\big]_{\partial \Omega}^2\right)+3\lambda_{\Gamma_D}\epsilon_{\Gamma}^{2}\left(\big[\mathcal{I}_D[u_1,u_2]\big]_{\Gamma}^2
+\big[\bm{\varphi}\big]_{\Gamma}^2\right)
+3\lambda_{\Gamma_N}\epsilon_{\Gamma}^{2}\left(\big[\mathcal{I}_N[u_1,u_2]\big]_{\Gamma}^2
+\big[\bm{\psi}\big]_{\Gamma}^2\right).
\end{aligned}
\end{equation*}
For $\mathbf{x}_{r_1}^i \in \mathcal{T}_{r_1}^{m_{r_1}}$,   $\mathbf{x}_{r_2}^i \in \mathcal{T}_{r_2}^{m_{r_2}}$, $\mathbf{x}^i_{b} \in \mathcal{T}_{b}^{m_{b}}$ and $\mathbf{x}^i_{\Gamma} \in \mathcal{T}_{\Gamma}^{m_{\Gamma}}$,
we denote the Voronoi cell associated with $\mathbf{x}_{r_1}^i$, $\mathbf{x}_{r_2}^i$, $\mathbf{x}_{b}^i$, $\mathbf{x}_{\Gamma}^i$ as $A_{\mathbf{x}_{r_1}^i}$, $A_{\mathbf{x}_{r_2}^i}$, $A_{\mathbf{x}_{b}^i}$ and $A_{\mathbf{x}_{\Gamma}^i}$ , respectively,
i.e.,
\begin{equation*}
\begin{aligned}
&A_{\mathbf{x}_{r_1}^i}= \{\mathbf{x} \in \Omega_1 \ \big| \ \norm{\mathbf{x} - \mathbf{x}_{r_1}^i}_2 = \min_{\mathbf{x}' \in \mathcal{T}_{r_1}^{m_{r_1}}} \norm{\mathbf{x} - \mathbf{x}'}_2 \}, \quad 	A_{\mathbf{x}_{r_2}^i}= \{x \in \Omega_2 \ \big| \ \norm{\mathbf{x} - \mathbf{x}_{r_2}^i}_2 = \min_{\mathbf{x}' \in \mathcal{T}_{r_2}^{m_{r_2}}} \norm{\mathbf{x} - \mathbf{x}'}_2 \},
\\
&A_{\mathbf{x}^i_{b}} = \{\mathbf{x} \in \partial\Omega \ \big| \ \norm{x - \mathbf{x}^i_{b}}_2 = \min_{x' \in \mathcal{T}_{b}^{m_{b}}} \norm{\mathbf{x} - \mathbf{x}'}_2 \},\quad 	A_{\mathbf{x}^i_{\Gamma}} = \{\mathbf{x} \in \Gamma \ \big| \ \norm{\mathbf{x} - \mathbf{x}^i_{\Gamma}}_2 = \min_{x' \in \mathcal{T}_{\Gamma}^{m_{\Gamma}}} \norm{\mathbf{x} - \mathbf{x}'}_2 \},
\end{aligned}
\end{equation*}
and let $\omega^i_{r_1} = \mu_{r_1}(A_{\mathbf{x}^i_{r_1}})$ , $\omega^i_{r_2} = \mu_{r_2}(A_{\mathbf{x}^i_{r_2}})$, $\omega^i_{b} = \mu_{b}(A_{\mathbf{x}^i_{b}})$ and $\omega^i_{\Gamma} = \mu_{\Gamma}(A_{\mathbf{x}^i_{\Gamma}})$. By taking the expectation with respect to $(\mathbf{x}_{r_1},\mathbf{x}_{r_2}, \mathbf{x}_b,\mathbf{x}_\Gamma) \sim \mu = \mu_{r_1}\times \mu_{r_2}\times \mu_{b}\times \mu_{\Gamma}$,
we obtain that
\begin{equation*}
\begin{aligned}
&\mathbb{E}_{\mu}[\mathbf{L}(\mathbf{x}_{r_1},\mathbf{x}_{r_2},\mathbf{x}_b,\mathbf{x}_\Gamma;u_1,u_2,\bm{\lambda},\bm{0})]\\
=&\sum_{i=1}^{m_{\Gamma}}\sum_{j=1}^{m_{b}}\sum_{k=1}^{m_{r_1}}\sum_{l=1}^{m_{r_2}}\int_{A_{\mathbf{x}_{\Gamma}^i}}\int_{A_{\mathbf{x}_{b}^j}}\int_{A_{\mathbf{x}_{r_1}^k}}\int_{A_{\mathbf{x}_{r_2}^l}}\mathbf{L}(\mathbf{x}_{r_1},\mathbf{x}_{r_2},\mathbf{x}_b,\mathbf{x}_\Gamma;u_1,u_2,\bm{\lambda},\bm{0}) d\mu\\
\leq&
3\sum_{i=1}^{m_{\Gamma}}\sum_{j=1}^{m_{b}}\sum_{k=1}^{m_{r_1}}\sum_{l=1}^{m_{r_2}}
\omega_{\Gamma}^i\omega_{b}^j\omega_{r_1}^k\omega_{r_2}^l
\mathbf{L}(\mathbf{x}_{r_1}^k,\mathbf{x}_{r_2}^l,\mathbf{x}_{b}^j,\mathbf{x}_{\Gamma}^i;u_1,u_2,\bm{\lambda},\bm{0})
+3\lambda_{r_1}\epsilon_{r_1}^{2}\left(\big[\mathcal{L}_1[u_1]\big]_{\Omega_1}^2+\big[f_1\big]_{\Omega_1}^2\right)\\
&+3\lambda_{r_2}\epsilon_{r_2}^{2}\left(\big[\mathcal{L}_2[u_2]\big]_{\Omega_2}^2+\big[f_2\big]_{\Omega_2}^2\right)
+3\lambda_{b}\epsilon_{b}^{2}\left(\big[\mathcal{B}[u_2]\big]_{\partial \Omega}^2+\big[\bm{g}\big]_{\partial \Omega}^2\right)\\
&+3\lambda_{\Gamma_D}\epsilon_{\Gamma}^{2}\left(\big[\mathcal{I}_D[u_1,u_2]\big]_{\Gamma}^2+\big[\bm{\varphi}\big]_{\Gamma}^2\right)
+3\lambda_{\Gamma_N}\epsilon_{\Gamma}^{2}\left(\big[\mathcal{I}_N[u_1,u_2]\big]_{\Gamma}^2
+\big[\bm{\psi}\big]_{\Gamma}^2\right),
\end{aligned}
\end{equation*}
where we have used the fact that $\sum_{i=1}^{m_{r_1}} \omega^i_{r_1} = 1$, $\sum_{i=1}^{m_{r_2}} \omega^i_{r_2} = 1$,
$\sum_{i=1}^{m_{b}} \omega^i_{b} = 1$ and $\sum_{i=1}^{m_{\Gamma}} \omega^i_{\Gamma} = 1$.

Next, we give the estimation of the first term on the right. Taking $\omega_{r_1}^{m_{r_1},*} = \max_i \omega_{r_1}^i$, $\omega_{r_2}^{m_{r_2},*} = \max_i \omega_{r_2}^i$, $\omega_{b}^{m_{b},*} = \max_i \omega^i_{b}$ and $\omega_{\Gamma}^{m_{\Gamma},*} = \max_i \omega^i_{\Gamma}$,
yields that
\begin{equation}\label{eq:proof of 4.1 eq}
\begin{aligned}
&3\sum_{i=1}^{m_{\Gamma}}\sum_{j=1}^{m_{b}}\sum_{k=1}^{m_{r_1}}\sum_{l=1}^{m_{r_2}}
\omega_{\Gamma}^i\omega_{b}^j\omega_{r_1}^k\omega_{r_2}^l
\mathbf{L}(\mathbf{x}_{r_1}^k,\mathbf{x}_{r_2}^l,\mathbf{x}_{b}^j,\mathbf{x}_{\Gamma}^i;u_1,u_2,\bm{\lambda},\bm{0})\\
\leq&
3m_{r_1} \omega_{r_1}^{m_{r_1},*}  \cdot \frac{\lambda_{r_1}}{m_{r_1}}\sum_{i=1}^{m_{r_1}} \left|\mathcal{L}_1[u_1](\mathbf{x}_{r_1}^i) - f_1(\mathbf{x}_{r_1}^i)\right|^2
+3m_{r_2} \omega_{r_2}^{m_{r_2},*}  \cdot \frac{\lambda_{r_2}}{m_{r_2}}\sum_{i=1}^{m_{r_2}} \left|\mathcal{L}_2[u_2](\mathbf{x}_{r_2}^i) - f_2(\mathbf{x}_{r_2}^i)\right|^2\\
& +3m_{\Gamma}\omega_{\Gamma}^{m_{\Gamma},*}
\cdot \left(\frac{\lambda_{\Gamma_D} }{m_{\Gamma}}\sum_{i=1}^{m_{\Gamma}} \norm{\mathcal{I}_D[u_1,u_2](\mathbf{x}_{\Gamma}^i) - \bm{\varphi}(\mathbf{x}_{\Gamma}^i)}_2^2+\frac{\lambda_{\Gamma_N} }{m_{\Gamma}}\sum_{i=1}^{m_{\Gamma}}\norm{\mathcal{I}_N[u_1,u_2](\mathbf{x}_{\Gamma}^i) - \bm{\psi}(\mathbf{x}_{\Gamma}^i)}_2^2\right)\\
& +3m_{b}\omega_{b}^{m_{b},*}
\cdot \frac{\lambda_{b} }{m_{b}}\sum_{i=1}^{m_{b}} \norm{\mathcal{B}[u_2](\mathbf{x}_{b}^i) - \bm{g}(\mathbf{x}_{b}^i)}_2^2,
\end{aligned}
\end{equation}
where we have used the fact that $m_{r_1}\omega_{r_1}^{m_{r_1},*}, m_{r_2}\omega_{r_2}^{m_{r_2},*}, m_{b}\omega_{b}^{m_{b},*}, m_{\Gamma}\omega_{\Gamma}^{m_{\Gamma},*} \ge 1$.

Let $B_\epsilon(\mathbf{x})$ be a closed ball centered at $\mathbf{x}$ with radius $\epsilon$. Let $P_{r_1}^* = \max_{\mathbf{x} \in \Omega_1} \mu_{r_1}(B_{\epsilon_{r_1}}(\mathbf{x}) \cap \Omega_1)$, $P_{r_2}^* = \max_{\mathbf{x} \in \Omega_2} \mu_{r_2}(B_{\epsilon_{r_2}}(\mathbf{x}) \cap \Omega_2)$, $P_{b}^* = \max_{\mathbf{x} \in \partial\Omega} \mu_{b}(B_{\epsilon_{b}}(\mathbf{x}) \cap \partial\Omega)$ and $P_{\Gamma}^* = \max_{\mathbf{x} \in \Gamma} \mu_{\Gamma}(B_{\epsilon_{\Gamma}}(\mathbf{x}) \cap \Gamma)$. Then for any $\mathbf{x}_{r_1} \in \Omega_1$, $\mathbf{x}_{r_2} \in \Omega_2$, $\mathbf{x}_{b} \in \partial\Omega$ and $\mathbf{x}_{\Gamma} \in \Gamma$, there exists $\mathbf{x}_{r_1}' \in \mathcal{T}_{r_1}$, $\mathbf{x}_{r_2}' \in \mathcal{T}_{r_2}$, $\mathbf{x}_{b}' \in \mathcal{T}_{b}$ and $\mathbf{x}_{\Gamma}' \in \mathcal{T}_{\Gamma}$ such that $\norm{\mathbf{x}_{r_1}- \mathbf{x}_{r_1}'}_2 \le \epsilon_{r_1}$, $\norm{\mathbf{x}_{r_2}- \mathbf{x}_{r_2}'}_2 \le \epsilon_{r_2}$, $\norm{\mathbf{x}_b - \mathbf{x}_b'}_2 \le \epsilon_b$ and $\norm{\mathbf{x}_{\Gamma} - \mathbf{x}_{\Gamma}'}_2 \le \epsilon_{\Gamma}$ for each $i$, there are closed balls $B_{\epsilon_{r_1}}$, $B_{\epsilon_{r_2}}$, $B_{\epsilon_{b}}$ and $B_{\epsilon_{\Gamma}}$ that include $A_{\mathbf{x}_{r_1}^i}$, $A_{\mathbf{x}_{r_2}^i}$, $A_{\mathbf{x}_{b}^i}$ and $A_{\mathbf{x}_{\Gamma}^i}$, respectively. These facts, together with Assumption \ref{assumption:data-dist} imply that
\begin{equation}\label{eq:proof of 4.1 eq2}
\omega_{r_1}^{m_{r_1},*} \le P_{r_1}^* \le C_{r_1}\epsilon_{r_1}^{d},\
\omega_{r_2}^{m_{r_2},*} \le P_{r_2}^* \le C_{r_2}\epsilon_{r_2}^{d},
\
\omega_{b}^{m_{b},*} \le P_{b}^* \le C_{b}\epsilon_{b}^{d-1},
\
\omega_{\Gamma}^{m_{\Gamma},*} \le P_{\Gamma}^* \le C_{\Gamma}\epsilon_{\Gamma}^{d-1}.
\end{equation}
With the estimations (\ref{eq:proof of 4.1 eq}) and (\ref{eq:proof of 4.1 eq2}), we obtain that
\begin{equation*}
\begin{aligned}
&\mathbb{E}_{\mu}[\mathbf{L}(\mathbf{x}_{r_1},\mathbf{x}_{r_2},\mathbf{x}_b,\mathbf{x}_\Gamma;u_1,u_2,\bm{\lambda},\bm{0})]
\\
\leq& 3C_{r_1}m_{r_1}\epsilon_{r_1}^{d}  \cdot \frac{\lambda_{r_1}}{m_{r_1}}\sum_{i=1}^{m_{r_1}} \left|\mathcal{L}_1[u_1](\mathbf{x}_{r_1}^i) - f_1(\mathbf{x}_{r_1}^i)\right|^2+	3C_{r_2}m_{r_2}\epsilon_{r_2}^{d}  \cdot \frac{\lambda_{r_2}}{m_{r_2}}\sum_{i=1}^{m_{r_2}} \left|\mathcal{L}_2[u_2](\mathbf{x}_{r_2}^i) - f_2(\mathbf{x}_{r_2}^i)\right|^2
\\
& +3C_\Gamma m_{\Gamma}\epsilon_{\Gamma}^{d-1}\cdot\left( \frac{\lambda_{\Gamma_D} }{m_{\Gamma}}\sum_{i=1}^{m_{\Gamma}} \norm{\mathcal{I}_D[u_1,u_2](\mathbf{x}_{\Gamma}^i) - \bm{\varphi}(\mathbf{x}_{\Gamma}^i)}_2^2 + \frac{\lambda_{\Gamma_N} }{m_{\Gamma}}\sum_{i=1}^{m_{\Gamma}} \norm{\mathcal{I}_N[u_1,u_2](\mathbf{x}_{\Gamma}^i) - \bm{\psi}(\mathbf{x}_{\Gamma}^i)}_2^2\right)
\\
& +3C_{b}m_{b}\epsilon_{b}^{d-1}
\cdot \frac{\lambda_{b} }{m_{b}}\sum_{j=1}^{m_{b}} \norm{\mathcal{B}[u_2](\mathbf{x}_{b}^j) - \bm{g}(\mathbf{x}_{b}^j)}_2^2
+3\lambda_{r_1}\epsilon_{r_1}^{2}\left(\big[\mathcal{L}_1[u_1]\big]_{\Omega_1}^2+\big[f_1\big]_{\Omega_1}^2\right)\\
&+3\lambda_{r_2}\epsilon_{r_2}^{2}\left(\big[\mathcal{L}_2[u_2]\big]_{\Omega_2}^2+\big[f_2\big]_{\Omega_2}^2\right)
+3\lambda_{b}\epsilon_{b}^{2}\left(\big[\mathcal{B}[u_2]\big]_{\partial \Omega}^2+\big[\bm{g}\big]_{\partial \Omega}^2\right)\\ &+3\lambda_{\Gamma_D}\epsilon_{\Gamma}^{2}\left(\big[\mathcal{I}_D[u_1,u_2]\big]_{\Gamma}^2
+\big[\bm{\varphi}\big]_{\Gamma}^2\right)
+3\lambda_{\Gamma_N}\epsilon_{\Gamma}^{2}\left(\big[\mathcal{I}_N[u_1,u_2]\big]_{\Gamma}^2
+\big[\bm{\psi}\big]_{\Gamma}^2\right)
.
\end{aligned}
\end{equation*}
Finally,  we conclude the proof by taking
$
C_{\bm{m}} = 3\max\{C_{r_1}m_{r_1}\epsilon_{r_1}^{d}, C_{r_2}m_{r_2}\epsilon_{r_2}^{d},
C_{b}m_{b}\epsilon_{b}^{d-1}, C_{\Gamma}m_{\Gamma}\epsilon_{\Gamma}^{d-1} \}
$.
\end{proof}

With Lemma \ref{app:lemma-lip} and Assumption \ref{assumption:data-dist}, we are able to  quantify the generalization error and provide an upper bound of the expected unregularized PINN loss (\ref{eq: the expected loss}).
\begin{lemma} \label{lem:gen}
Suppose Assumption \ref{assumption:data-dist} holds. Suppose that $u_1, u_2$ satisfy
\begin{equation*}
R_{r_1}(u_1)  < \infty, \
R_{r_2}(u_2)  < \infty, \
R_{b}(u_2) < \infty,  \
R_{\Gamma_D}(u_1,u_2)  < \infty, \
R_{\Gamma_N}(u_1,u_2)  < \infty,
\end{equation*}
and $f_1, f_2, \bm{\psi}, \bm{\varphi}, \bm{g}$ satisfy
\begin{equation*}
\big[f_1\big]_{\Omega_1},\  \big[f_2\big]_{\Omega_2},\ \big[\bm{g}\big]_{\partial\Omega}, \ \big[\bm{\varphi}\big]_{\Gamma},\  \big[\bm{\psi}\big]_{\Gamma}  < \infty.
\end{equation*}
Let $m_{r_1}$, $m_{r_2}$, $m_b$ and $m_{\Gamma}$ be the number of iid samples from $\mu_{r_1}$, $\mu_{r_2}$, $\mu_{b}$ and $\mu_{\Gamma}$, respectively. Let $\bm{\lambda} = (\lambda_{r_1},\lambda_{r_2},\lambda_{b},\lambda_{
\Gamma_D},\lambda_{\Gamma_N})$ be a fixed vector. Then, with probability at least,
\begin{equation*}
P = P(m_{r_1})P(m_{r_2})P(m_{b})P(m_{\Gamma}), \quad\text{where } P(m) = (1 - \sqrt{m}(1-1/\sqrt{m})^{m}),
\end{equation*}
we have
\begin{equation*}
\text{Loss}^{\text{PINN}}(u_1,u_2;\bm{\lambda})
\le C_{\bm{m}}\cdot \text{Loss}_{\bm{m}}(u_1,u_2;\bm{\lambda}, \bm{\hat{\lambda}}_{\bm{m}}^R)
+C'(m_{r_1}^{-\frac{1}{d}} +  m_{r_2}^{-\frac{1}{d}} + m_{b}^{-\frac{1}{d-1}}+m_{\Gamma}^{-\frac{1}{d-1}}).
\end{equation*}
Here, $\bm{\hat{\lambda}}_{\bm{m}}^R = (\hat{\lambda}_{r_1,\bm{m}}^R,\hat{\lambda}_{r_2,\bm{m}}^R,\hat{\lambda}_{b,\bm{m}}^R,\hat{\lambda}_{\Gamma_D,\bm{m}}^R, \hat{\lambda}_{\Gamma_N,\bm{m}}^R)$. Specifically,
\begin{equation*}
\begin{aligned}
    &\hat{\lambda}_{r_1,\bm{m}}^R = \frac{3\lambda_{r_1} dc_{r_1}^{-\frac{2}{d}}}{C_{\bm{m}}}\cdot m_{r_1}^{-\frac{1}{d}},
    \quad
    \hat{\lambda}_{r_2,\bm{m}}^R = \frac{3\lambda_{r_2} dc_{r_2}^{-\frac{2}{d}}}{C_{\bm{m}}}\cdot m_{r_2}^{-\frac{1}{d}},
    \quad \hat{\lambda}_{b,\bm{m}}^R = \frac{3\lambda_{b} d c_{b}^{-\frac{2}{d-1}} }{C_{\bm{m}}}\cdot m_{b}^{-\frac{1}{d-1}},\\
    &\hat{\lambda}_{\Gamma_D,\bm{m}}^R = \frac{3\lambda_{\Gamma_D}d c_{\Gamma}^{-\frac{2}{d-1}} }{C_{\bm{m}}}\cdot m_{\Gamma}^{-\frac{1}{d-1}},
    \quad \hat{\lambda}_{\Gamma_N,\bm{m}}^R = \frac{3\lambda_{\Gamma_N} d c_{\Gamma}^{-\frac{2}{d-1}} }{C_{\bm{m}}}\cdot m_{\Gamma}^{-\frac{1}{d-1}}.
    \end{aligned}
\end{equation*}
$C_{\bm{m}} = 3\max\{\kappa_{r_1} \sqrt{d}^{d} m_{r_1}^{\frac{1}{2}}, \kappa_{r_2} \sqrt{d}^{d} m_{r_2}^{\frac{1}{2}}, \kappa_{b} \sqrt{d}^{d-1} m_{b}^{\frac{1}{2}}, \kappa_{\Gamma} \sqrt{d}^{d-1} m_{\Gamma}^{\frac{1}{2}}\}$
where $\kappa_{r_1} = \frac{C_{r_1}}{c_{r_1}}$, $\kappa_{r_2} = \frac{C_{r_2}}{c_{r_2}}$, $\kappa_{b} = \frac{C_{b}}{c_{b}}$, $\kappa_{\Gamma} = \frac{C_{\Gamma}}{c_{\Gamma}}$. And
$C'$ is a constant that depends only on $\bm{\lambda}$, $d$, $c_{r_1}$, $c_{r_2}$, $c_{b}$, $c_{\Gamma}$, $f_1$, $f_2$, $\bm{g}$, $\bm{\varphi}$, $\bm{\psi}$.
\end{lemma}
\begin{proof}
Since $\mathcal{T}_{r_1} = \{\mathbf{x}_{r_1}^i\}_{i=1}^{m_{r_1}}$ be iid samples from $\mu_{r_1}$ on $\Omega_1$, $\mathcal{T}_{r_2} = \{\mathbf{x}_{r_2}^i\}_{i=1}^{m_{r_2}}$ be iid samples from $\mu_{r_2}$ on $\Omega_2$, $\mathcal{T}_{b} = \{\mathbf{x}_{b}^i\}_{i=1}^{m_{b}}$ be iid samples from $\mu_{b}$ on $\partial\Omega$
and $\mathcal{T}_{\Gamma} = \{\mathbf{x}_{\Gamma}^i\}_{i=1}^{m_{\Gamma}}$ be iid samples from $\mu_{\Gamma}$ on $\Gamma$, respectively, therefore, by Lemma B.2 in \cite{shin2020on}, with probability at least
\begin{equation} \label{app:thm-prob}
P = P(m_{r_1})P(m_{r_2})P(m_{b})P(m_{\Gamma}), \quad\text{where } P(m) = (1 - \sqrt{m}(1-1/\sqrt{m})^{m}),
\end{equation}
for $\forall \mathbf{x}_{r_1} \in \Omega_1$, $\forall \mathbf{x}_{r_2} \in \Omega_2$, $\forall \mathbf{x}_b \in \partial\Omega$ and $\forall \mathbf{x}_\Gamma \in \Gamma$,
there exists $\mathbf{x}_{r_1}' \in \mathcal{T}_{r_1}^{m_{r_1}}$, $\mathbf{x}_{r_2}' \in \mathcal{T}_{r_2}^{m_{r_2}}$, $\mathbf{x}_{b}' \in \mathcal{T}_{b}^{m_b}$ and $\mathbf{x}_{\Gamma}' \in \mathcal{T}_{\Gamma}^{m_\Gamma}$ such that
$\norm{\mathbf{x}_{r_1}- \mathbf{x}_{r_1}'}_2  \le \sqrt{d}c_{r_1}^{-\frac{1}{d}}m_{r_1}^{-\frac{1}{2d}}$, $\norm{\mathbf{x}_{r_2}- \mathbf{x}_{r_2}'}_2  \le \sqrt{d}c_{r_2}^{-\frac{1}{d}}m_{r_2}^{-\frac{1}{2d}}$, $\norm{\mathbf{x}_b - \mathbf{x}_b' }_2\le \sqrt{d}c^{-\frac{1}{d-1}}_{b}m_{b}^{-\frac{1}{2(d-1)}}$ and
$\norm{\mathbf{x}_{\Gamma} - \mathbf{x}_{\Gamma}'}_2 \le \sqrt{d}c^{-\frac{1}{d-1}}_{\Gamma}m_{\Gamma}^{-\frac{1}{2(d-1)}}$.

Using Lemma \ref{app:lemma-lip}, together with taking $\epsilon_{r_1} = \sqrt{d}c_{r_1}^{-\frac{1}{d}}m_{r_1}^{-\frac{1}{2d}}$, $\epsilon_{r_2} = \sqrt{d}c_{r_2}^{-\frac{1}{d}}m_{r_2}^{-\frac{1}{2d}}$, $\epsilon_{b} = \sqrt{d}c^{-\frac{1}{d-1}}_{b}m_{b}^{-\frac{1}{2(d-1)}}$
and $\epsilon_{\Gamma} = \sqrt{d}c^{-\frac{1}{d-1}}_{\Gamma}m_{\Gamma}^{-\frac{1}{2(d-1)}}$, implies
that with probability at least \eqref{app:thm-prob},
\begin{equation*}
\begin{aligned}
\text{Loss}^{\text{PINN}}(u_1,u_2;\bm{\lambda})
\leq&
C_{\bm{m}} \cdot \text{Loss}_{\bm{m}}^{\text{PINN}}(u_1,u_2;\bm{\lambda})+ Q\\
& +C_{\bm{m}} \cdot \left[\hat{\lambda}_{r_1,\bm{m}}^R \cdot \big[\mathcal{L}_1[u_1]\big]_{\Omega_1}^2
+\hat{\lambda}_{r_2,\bm{m}}^R \cdot \big[\mathcal{L}_2[u_2]\big]_{\Omega_2}^2
+ \hat{\lambda}_{b,\bm{m}}^R \cdot \big[\mathcal{B}[u_2]\big]_{\partial\Omega}^2
\right] \\
& +C_{\bm{m}} \cdot\left[\hat{\lambda}_{\Gamma_D,\bm{m}}^R \cdot \big[\mathcal{I}_D[u_1,u_2]\big]_{\Gamma}^2 + \hat{\lambda}_{\Gamma_N,\bm{m}}^R \cdot \big[\mathcal{I}_N[u_1,u_2]\big]_{\Gamma}^2\right] ,
\end{aligned}
\end{equation*}
where
\begin{equation*}
\begin{aligned}
Q =&3\lambda_{r_1} d c_{r_1}^{-\frac{2}{d}}m_{r_1}^{-\frac{1}{d}}\big[f_1\big]_{\Omega_1}^2
+ 3\lambda_{r_2} d c_{r_2}^{-\frac{2}{d}}m_{r_2}^{-\frac{1}{d}}\big[f_2\big]_{\Omega_2}^2  +3\lambda_{b} d c_{b}^{-\frac{2}{d-1}}m_{b}^{-\frac{1}{d-1}}\big[\bm{g}\big]_{\partial\Omega}^2\\
&+3 d c_{\Gamma}^{-\frac{2}{d-1}}m_{\Gamma}^{-\frac{1}{d-1}}\left(\lambda_{\Gamma_D}\big[\bm{\varphi}\big]_{\Gamma}^2 +\lambda_{\Gamma_N}\big[\bm{\psi}\big]_{\Gamma}^2\right),
\end{aligned}
\end{equation*}
and
\begin{equation*}
\begin{aligned}
    &C_{\bm{m}} = 3\max\{\frac{C_{r_1}}{c_{r_1}}\sqrt{d}^{d} m_{r_1}^{\frac{1}{2}},  \frac{C_{r_2}}{c_{r_2}}\sqrt{d}^{d} m_{r_2}^{\frac{1}{2}}, \frac{C_{b}}{c_{b}} \sqrt{d}^{d-1} m_{b}^{\frac{1}{2}}, \frac{C_{\Gamma}}{c_{\Gamma}} \sqrt{d}^{d-1} m_{\Gamma}^{\frac{1}{2}}\},\\
    &\hat{\lambda}_{r_1,\bm{m}}^R = \frac{3\lambda_{r_1} dc_{r_1}^{-\frac{2}{d}}}{C_{\bm{m}}}\cdot m_{r_1}^{-\frac{1}{d}},
    \quad
    \hat{\lambda}_{r_2,\bm{m}}^R = \frac{3\lambda_{r_2} dc_{r_2}^{-\frac{2}{d}}}{C_{\bm{m}}}\cdot m_{r_2}^{-\frac{1}{d}},
    \quad
    \hat{\lambda}_{b,\bm{m}}^R = \frac{3\lambda_{b} d c_{b}^{-\frac{2}{d-1}} }{C_{\bm{m}}}\cdot m_{b}^{-\frac{1}{d-1}},\\
    &\hat{\lambda}_{\Gamma_D,\bm{m}}^R = \frac{3\lambda_{\Gamma_D} d c_{\Gamma}^{-\frac{2}{d-1}} }{C_{\bm{m}}}\cdot m_{\Gamma}^{-\frac{1}{d-1}},
    \quad \hat{\lambda}_{\Gamma_N,\bm{m}}^R = \frac{3\lambda_{\Gamma_N} d c_{\Gamma}^{-\frac{2}{d-1}} }{C_{\bm{m}}}\cdot m_{\Gamma}^{-\frac{1}{d-1}}.
    \end{aligned}
\end{equation*}
By taking
\begin{equation*}
\begin{aligned}
C' = 3\max &\{\lambda_{r_1} d c_{r_1}^{-\frac{2}{d}}\big[f_1\big]_{\Omega_1}^2, \lambda_{r_2} d c_{r_2}^{-\frac{2}{d}}\big[f_2\big]_{\Omega_2}^2, \lambda_{b}d c_{b}^{-\frac{2}{d-1}}\big[\bm{g}\big]_{\partial\Omega}^2, d c_{\Gamma}^{-\frac{2}{d-1}}\left(\lambda_{\Gamma_D}\big[\bm{\varphi}\big]_{\Gamma}^2 +\lambda_{\Gamma_N}\big[\bm{\psi}\big]_{\Gamma}^2 \right) \},
\end{aligned}
\end{equation*}
we conclude that
\begin{equation*}
\text{Loss}^{\text{PINN}}(u_1,u_2;\bm{\lambda})
\le C_{\bm{m}}\cdot \text{Loss}_{\bm{m}}(u_1,u_2;\bm{\lambda}, \bm{\hat{\lambda}}_{\bm{m}}^R)
+C'(m_{r_1}^{-\frac{1}{d}} +  m_{r_2}^{-\frac{1}{d}} + m_{b}^{-\frac{1}{d-1}}+m_{\Gamma}^{-\frac{1}{d-1}}),
\end{equation*}
where $\bm{\hat{\lambda}}_{\bm{m}}^R = (\hat{\lambda}_{r_1,\bm{m}}^R,\hat{\lambda}_{r_2,\bm{m}}^R,\hat{\lambda}_{b,\bm{m}}^R,\hat{\lambda}_{\Gamma_D,\bm{m}}^R, \hat{\lambda}_{\Gamma_N,\bm{m}}^R)$. The proof is completed.
\end{proof}

Using Lemma \ref{lem:gen}, we will show that the expected PINN loss \eqref{eq: the expected loss} at the minimizers of the Lipschitz regularized empirical loss \eqref{def: GE LIP PINN empirical loss} converges to zero according to Assumptions \ref{assumption:convergence}.

\begin{lemma} \label{lem:conv-loss}
Suppose Assumptions \ref{assumption:data-dist} and  \ref{assumption:convergence} hold. Let $m_{r_1}$, $m_{r_2}$, $m_{\Gamma}$ and $m_{b}$ be the number of iid samples from
$\mu_{r_1}$, $\mu_{r_2}$, $\mu_{\Gamma}$ and $\mu_{b}$, respectively,
and satisfy $m_{r_2} = \mathcal{O}(m_{r_1})$, $m_{\Gamma} = \mathcal{O}(m_{r_1}^{\frac{d-1}{d}})$, $m_{b} = \mathcal{O}(m_{r_1}^{\frac{d-1}{d}})$.
Let $\bm{\lambda}_{\bm{m}}^R$ be a vector satisfying
\begin{equation*}
    \bm{\lambda}_{\bm{m}}^R \ge \bm{\hat{\lambda}}_{\bm{m}}^R, \qquad
    \norm{\bm{\lambda}_{\bm{m}}^R}_\infty = \mathcal{O}(\norm{\bm{\hat{\lambda}}_{\bm{m}}^R}_\infty),
\end{equation*}
where $\bm{\hat{\lambda}}_{\bm{m}}^R = (\hat{\lambda}_{r_1,\bm{m}}^R,\ \hat{\lambda}_{r_2,\bm{m}}^R,\ \hat{\lambda}_{b,\bm{m}}^R, \ \hat{\lambda}_{\Gamma_D,\bm{m}}^R,\ \hat{\lambda}_{\Gamma_N,\bm{m}}^R)$ are defined in Lemma \ref{lem:gen}. Let $(u_{1,\bm{m}},u_{2,\bm{m}}) \in (\mathcal{H}_{1,\bm{m}}, \mathcal{H}_{2,\bm{m}})$ be a minimizer of the Lipschitz regularized empirical loss $\text{Loss}_{m}(\cdot;\bm{\lambda}, \bm{\lambda}_{\bm{m}}^R)$ \eqref{def: GE LIP PINN empirical loss}.
Then the following holds:
\begin{itemize}
\item With probability at least
$$P = P(m_{r_1})P(m_{r_2})P(m_{b})P(m_{\Gamma}), \quad\text{where } P(m) = (1 - \sqrt{m}(1-1/\sqrt{m})^{m})$$ over iid samples,
\begin{equation*}
\text{Loss}^{\text{PINN}}(u_{1,\bm{m}},u_{2,\bm{m}};\bm{\lambda}) = \mathcal{O}(m_{r_1}^{-\frac{1}{d}}).
\end{equation*}
\item With probability 1 over iid samples,
\begin{equation*}
\begin{aligned}\label{thm:l2-convergence}
&\lim_{m_{r_1} \to \infty} \mathcal{L}[u_{1,\bm{m}}] = f_1 \text{ in } L^2(\Omega_1),
\quad
\lim_{m_{r_1} \to \infty} \mathcal{L}[u_{2,\bm{m}}] = f_2 \text{ in } L^2(\Omega_2),
\quad
\lim_{m_{r_1} \to \infty} \mathcal{B}[u_{2,\bm{m}}] = \bm{g} \text{ in } L^2(\partial \Omega),
\\	
&\lim_{m_{r_1} \to \infty} \mathcal{I}_D[u_{1, \bm{m}},u_{2, \bm{m}}] = \bm{\varphi} \text{ in } L^2(\Gamma),
\quad
\lim_{m_{r_1} \to \infty} \mathcal{I}_N[u_{1, \bm{m}},u_{2, \bm{m}}] = \bm{\psi} \text{ in } L^2(\Gamma).
\end{aligned}
\end{equation*}
\end{itemize}
\end{lemma}
\begin{proof}
Since $m_{r_1}=\mathcal{O}(m_{r_2})= \mathcal{O}(m_{b}^{\frac{d}{d-1}})=\mathcal{O}(m_{\Gamma}^{\frac{d}{d-1}})$, we have
\begin{equation*}
\hat{\lambda}_{r_1,\bm{m}}^R,\ \hat{\lambda}_{r_2,\bm{m}}^R,\ \hat{\lambda}_{b,\bm{m}}^R, \ \hat{\lambda}_{\Gamma_D,\bm{m}}^R,\ \hat{\lambda}_{\Gamma_N,\bm{m}}^R =\mathcal{O}(m_{r_1}^{-\frac{1}{2}-\frac{1}{d}}),
\end{equation*}
where $\hat{\lambda}_{r_1,\bm{m}}^R,\ \hat{\lambda}_{r_2,\bm{m}}^R,\ \hat{\lambda}_{b,\bm{m}}^R, \ \hat{\lambda}_{\Gamma_D,\bm{m}}^R,\ \hat{\lambda}_{\Gamma_N,\bm{m}}^R$ are defined in Lemma \ref{lem:gen}. Let $\bm{\lambda}$ be a vector independent of $\bm{m}$
and $\bm{\lambda}_{\bm{m}}^R = (\lambda_{r_1,\bm{m}}^R,\ \lambda_{r_2,\bm{m}}^R,\
\lambda_{b,\bm{m}}^R, \ \lambda_{\Gamma_D,\bm{m}}^R,\ \lambda_{\Gamma_N,\bm{m}}^R)$
be a vector satisfying
\begin{equation*}
    \bm{\lambda}_{\bm{m}}^R \ge \bm{\hat{\lambda}}_{\bm{m}}^R, \qquad
    \norm{\bm{\lambda}_{\bm{m}}^R}_\infty = \mathcal{O}\left(\norm{\bm{\hat{\lambda}}_{\bm{m}}^R}_\infty\right),
\end{equation*}
where $\bm{\hat{\lambda}}_{\bm{m}}^R = (\hat{\lambda}_{r_1,\bm{m}}^R,\ \hat{\lambda}_{r_2,\bm{m}}^R,\ \hat{\lambda}_{b,\bm{m}}^R, \ \hat{\lambda}_{\Gamma_D,\bm{m}}^R,\ \hat{\lambda}_{\Gamma_N,\bm{m}}^R)$. Let $(u_{1,\bm{m}},u_{2,\bm{m}}) \in (\mathcal{H}_{1,\bm{m}}, \mathcal{H}_{2,\bm{m}})$ minimizes the Lipschitz regularized loss $\text{Loss}_{m}(\cdot;\bm{\lambda}, \bm{\lambda}_{\bm{m}}^R)$ \eqref{def: GE LIP PINN empirical loss}. Let $(\hat{u}_{1,\bm{m}},\hat{u}_{2,\bm{m}})$ be the neural networks defined in the third term of Assumption~\ref{assumption:convergence}, i.e., they satisfy $\text{Loss}_{\bm{m}}^{\text{PINN}}(\hat{u}_{1,\bm{m}},\hat{u}_{2,\bm{m}};\bm{\lambda}) = \mathcal{O}(m_{r_1}^{-\frac{1}{2}-\frac{1}{d}})$.

Then, we have
\begin{equation*}
\begin{aligned}
\text{Loss}_{\bm{m}}(u_{1,\bm{m}},u_{2,\bm{m}};\bm{\lambda}, \bm{\lambda}_{\bm{m}}^R)
\leq& \text{Loss}_{\bm{m}}(\hat{u}_{1,\bm{m}},\hat{u}_{2,\bm{m}};\bm{\lambda}, \bm{\lambda}_{\bm{m}}^R)
\\
\leq& \norm{\bm{\lambda}_{\bm{m}}^R}_\infty\big(R_{r_1}(\hat{u}_{1,\bm{m}}) + R_{r_2}(\hat{u}_{2,\bm{m}}) + R_{\Gamma_D}(\hat{u}_{2,\bm{m}}, \hat{u}_{1,\bm{m}}) + R_{\Gamma_N}(\hat{u}_{2,\bm{m}}, \hat{u}_{1,\bm{m}}) \\  &+R_{b}(u_{2,\bm{m}})\big)
+ \text{Loss}_{\bm{m}}(\hat{u}_{1,\bm{m}},\hat{u}_{2,\bm{m}};\bm{\lambda}, 0).
\end{aligned}
\end{equation*}
Let
\begin{equation*}
\hat{R} = \sup_{\bm{m}} \left(R_{r_1}(\hat{u}_{1,\bm{m}}) + R_{r_2}(\hat{u}_{2,\bm{m}}) + R_{\Gamma_D}(\hat{u}_{1,\bm{m}}, \hat{u}_{2,\bm{m}}) + R_{\Gamma_N}(\hat{u}_{1,\bm{m}}, \hat{u}_{2,\bm{m}}) + R_{b}(\hat{u}_{2,\bm{m}}) \right).
\end{equation*}
Note that $\text{Loss}_{\bm{m}}^{\text{PINN}}(\hat{u}_{1,\bm{m}},\hat{u}_{2,\bm{m}};\bm{\lambda}) =\text{Loss}_{\bm{m}}(\hat{u}_{1,\bm{m}},\hat{u}_{2,\bm{m}};\bm{\lambda}, 0)$ and $\norm{\bm{\lambda}_{\bm{m}}^R}_\infty = \mathcal{O}(\norm{\bm{\hat{\lambda}}_{\bm{m}}^R}_\infty) = \mathcal{O}(m_{r_1}^{-\frac{1}{2}-\frac{1}{d}})$. By the last term in Assumption \ref{assumption:convergence}, we have $\hat{R} < \infty$.
Therefore,
\begin{equation*}
    \text{Loss}_{\bm{m}}(u_{1,\bm{m}}, u_{2,\bm{m}};\bm{\lambda},\bm{\lambda}_{\bm{m}}^R) = \mathcal{O}(m_{r_1}^{-\frac{1}{2}-\frac{1}{d}}).
\end{equation*}

According to Lemma \ref{lem:gen}, with probability at least
\begin{equation*}
P = P(m_{r_1})P(m_{r_2})P(m_{b})P(m_{\Gamma}), \quad\text{where } P(m) = (1 - \sqrt{m}(1-1/\sqrt{m})^{m}),
\end{equation*}
we have that
\begin{equation*}
\begin{aligned}
\text{Loss}^{\text{PINN}}(u_{1,\bm{m}},u_{2,\bm{m}};\bm{\lambda}) \leq& C_{\bm{m}}\cdot \text{Loss}_{\bm{m}}(u_{1,\bm{m}},u_{2,\bm{m}};\bm{\lambda}, \bm{\hat{\lambda}}_{\bm{m}}^R)
+C'(m_{r_1}^{-\frac{1}{d}} +  m_{r_2}^{-\frac{1}{d}} + m_{b}^{-\frac{1}{d-1}}+m_{\Gamma}^{-\frac{1}{d-1}})
=\mathcal{O}(m_{r_1}^{-\frac{1}{d}}),
\end{aligned}
\end{equation*}
which completes the first part of the proof. Here, we have used the fact that $m_{r_2} = \mathcal{O}(m_{r_1})$, $m_{\Gamma} = \mathcal{O}(m_{r_1}^{\frac{d-1}{d}})$, $m_{b} = \mathcal{O}(m_{r_1}^{\frac{d-1}{d}})$ and $C_{\bm{m}}=\mathcal{O}(m_{r_1}^{\frac{1}{2}})$.

In addition, by the first part of the Lemma,  we have that the probability of $$\lim_{m_{r_1} \to \infty} \text{Loss}(u_{1,\bm{m}},u_{2,\bm{m}};\bm{\lambda},\bm{0}) = 0$$ is one. Consequently, with probability one over iid samples,
\begin{align*}
0 &=\lim_{m_{r_1} \to \infty} \text{Loss}(u_{1,\bm{m}},u_{2,\bm{m}};\bm{\lambda},\bm{0})\\
&= \lim_{m_{r_1} \to \infty} \Big(\lambda_{r_1} \int_{\Omega_1} \left|\mathcal{L}_1[u_{1,{\bm{m}}}](\mathbf{x}_{r_1}) - f_1(\mathbf{x}_{r_1}) \right|^2d\mu_{r_1}(\mathbf{x}_{r_1}) +  \lambda_{r_2} \int_{\Omega_2} \left|\mathcal{L}_2[u_{2,{\bm{m}}}](\mathbf{x}_{r_2}) - f_2(\mathbf{x}_{r_2}) \right|^2d\mu_{r_2}(\mathbf{x}_{r_2})\\
&\qquad\quad + \lambda_{b}\int_{\partial\Omega} \norm{\mathcal{B}[u_{2,\bm{m}}](\mathbf{x}_{b}) - \bm{g}(\mathbf{x}_{b}) }_2^2 d\mu_{b}(\mathbf{x}_{b})+\lambda_{\Gamma_D}\int_{\Gamma} \norm{\mathcal{I}_D[u_{1,\bm{m}},u_{2,\bm{m}}](\mathbf{x}_{\Gamma}) - \bm{\varphi}(\mathbf{x}_{\Gamma}) }_2^2 d\mu_{\Gamma}(\mathbf{x}_{\Gamma})\\
&\qquad\quad +\lambda_{\Gamma_N}\int_{\Gamma} \norm{\mathcal{I}_N[u_{1,\bm{m}},u_{2,\bm{m}}](\mathbf{x}_{\Gamma}) - \bm{\psi}(\mathbf{x}_{\Gamma}) }_2^2 d\mu_{\Gamma}(\mathbf{x}_{\Gamma})\Big).
\end{align*}
Since $\bm{\lambda} = (\lambda_{r_1}, \lambda_{r_2}, \lambda_{b}, \lambda_{\Gamma_D}, \lambda_{\Gamma_N})\geq 0$, we obtain that
\begin{equation*}
\begin{aligned}
&\lim_{m_{r_1} \to \infty} \int_{\Omega_1} \left|\mathcal{L}_1[u_{1,{\bm{m}}}](\mathbf{x}_{r_1}) - f_1(\mathbf{x}_{r_1}) \right|^2d\mu_{r_1}(\mathbf{x}_{r_1}) = 0,
\
\lim_{m_{r_1} \to \infty}\int_{\Omega_2} \left|\mathcal{L}_2[u_{2,{\bm{m}}}](\mathbf{x}_{r_2}) - f_2(\mathbf{x}_{r_2}) \right|^2d\mu_{r_2}(\mathbf{x}_{r_2})=0,\\
&\lim_{m_{r_1} \to \infty} \int_{\partial\Omega} \norm{\mathcal{B}[u_{2,\bm{m}}](\mathbf{x}_{b}) - \bm{g}(\mathbf{x}_{b}) }_2^2 d\mu_{b}(\mathbf{x}_{b})=0,
\
\lim_{m_{r_1} \to \infty} \int_{\Gamma} \norm{\mathcal{I}_D[u_{1,\bm{m}},u_{2,\bm{m}}](\mathbf{x}_{\Gamma}) - \bm{\varphi}(\mathbf{x}_{\Gamma}) }_2^2 d\mu_{\Gamma}(\mathbf{x}_{\Gamma})=0,\\
&\lim_{m_{r_1} \to \infty} \int_{\Gamma} \norm{\mathcal{I}_N[u_{1,\bm{m}},u_{2,\bm{m}}](\mathbf{x}_{\Gamma}) - \bm{\psi}(\mathbf{x}_{\Gamma}) }_2^2 d\mu_{\Gamma}(\mathbf{x}_{\Gamma})=0.
\end{aligned}
\end{equation*}
Therefore, we conclude that $\mathcal{L}_1[u_{1,m}] \to f_1$ in $L^2(\Omega_1;\mu_{r_1})$, $\mathcal{L}_2[u_{2,m}] \to f_2$ in $L^2(\Omega_2;\mu_{r_2})$, $\mathcal{B}[u_{2,m}] \to \bm{g}$ in $L^2(\partial\Omega;\mu_{b})$, $\mathcal{I}_D[u_{1,m},u_{2,m}] \to \bm{\varphi}$ in $L^2(\Gamma;\mu_{\Gamma})$ and $\mathcal{I}_N[u_{1,m},u_{2,m}] \to \bm{\psi}$
in $L^2(\Gamma;\mu_{\Gamma})$ as $m_{r_1} \to \infty$.
\end{proof}

Finally, to complete the proof, it is sufficient to present the following estimate for the interface problem \eqref{eq:interface problem}. For convenience, we denote
 $X=H^2(\Omega_1)\cap H^2(\Omega_2)$
and define
$$\norm{u}_X = \norm{u}_{H^2(\Omega_1)} + \norm{u}_{H^2(\Omega_2)},\  \forall u\in \ X.$$
\begin{lemma}\label{lem:bound of solu}
Assume that $\varphi \in H^{2}(\Gamma)$, $\ \psi \in H^{1}(\Gamma)$, $g \ \in H^{2}(\partial \Omega)$, $f_i\in L^2(\Omega), \ i=1,2$. Then the problem (\ref{eq:interface problem}) has a unique solution $u\in X$ and $u$ satisfies the estimate:
\begin{equation*}
\norm{u}_X \leq C \left(\norm{f}_{L^2(\Omega)} + \norm{g}_{H^{2}(\partial \Omega)} + \norm{\varphi}_{H^{2}(\Gamma)} + \norm{\psi}_{H^{1}(\Gamma)}\right).
\end{equation*}
\end{lemma}
\begin{proof}
Let $\tilde{u}_1$ solve
\begin{equation*}
\begin{aligned}
-\Delta \tilde{u}_1=0,\ \text{in}\ \Omega_1,\\
\ \tilde{u}_1= -\varphi\ \text{on}\ \Gamma.
\end{aligned}
\end{equation*}
We know $\tilde{u}_1$ exists and  $\tilde{u}_1\in H^2(\Omega_1)$ satisfying (cf. Grisvard \cite{grisvard2011elliptic})
\begin{equation*}
\norm{\tilde{u}_1}_{H^2(\Omega_1)}\leq c \norm{\varphi}_{H^{3/2}(\Gamma)}.
\end{equation*}
Let $\tilde{u}_2$ solve
\begin{equation*}
\begin{aligned}
&-\Delta^2\tilde{u}_2 = 0,\ \text{in}\ \Omega_2,\\
& \tilde{u}_2 = 0,\ \frac{\partial \tilde{u}_2}{\partial n} = \frac{a_1}{a_2} \frac{\partial \tilde{u}_1}{\partial n} +\frac{\psi}{a_2} \ \text{on}\ \Gamma, \\
& \tilde{u}_2 = g,\ \frac{\partial \tilde{u}_2}{\partial \nu} =0\ \text{on}\ \partial \Omega.
\end{aligned}
\end{equation*}
We know $\tilde{u}_2$ exists and  $\tilde{u}_2\in H^2(\Omega_2)$ satisfying (cf. Girault-Raviart \cite{girault2012finite}, pp.15-17)
\begin{equation*}
\begin{aligned}
\norm{\tilde{u}_2}_{H^2(\Omega_2)} \leq& C\left(\norm{g}_{H^{3/2}(\partial\Omega)} + \norm{\frac{\partial \tilde{u}_1}{\partial n}}_{H^{1/2}(\Gamma)} + \norm{\psi}_{H^{1/2}(\Gamma)}\right)\\
\leq& C\left(\norm{g}_{H^{3/2}(\partial\Omega)} +  \norm{\tilde{u}_1}_{H^2(\Omega_1)}+ \norm{\psi}_{H^{1/2}(\Gamma)}\right)\\
\leq & C\left(\norm{g}_{H^{3/2}(\partial\Omega)} + \norm{\varphi}_{H^{3/2}(\Gamma)}+ \norm{\psi}_{H^{1/2}(\Gamma)}\right),
\end{aligned}
\end{equation*}
where $C$ is a generic constant that depends on $\Omega, a_1, a_2$. Let
\begin{equation*}
\tilde{u}(x) = \left\{\begin{aligned}
&\tilde{u}_1(x), \ x\in \Omega_1,\\
&\tilde{u}_2(x), \ x\in \Omega_2.\\
\end{aligned}\right.
\end{equation*}
Obviously, $\tilde{u}(x)\in X$. In addition, by \cite{babuvska1970finite,bruce1974poisson} we know that the equation
\begin{equation*}
\begin{aligned}
 -\nabla\cdot(a_i\nabla v)+b_i v &=  f_i +\nabla\cdot(a_i\nabla \tilde{u}_i)-b_i \tilde{u}_i,\quad \text{in} \ \Omega_i, \ i=1,2,\\
 \llbracket a\nabla v\cdot \mathbf{n} \rrbracket &=0, \quad  \text{on} \ \Gamma , \\
\llbracket v\rrbracket &=0,\quad  \text{on} \ \Gamma ,\\
v&=0,\quad \text{on} \ \partial\Omega,
\end{aligned}
\end{equation*}
has a unique solution $v\in X$ and $v$ satisfies the estimate
\begin{equation*}
\begin{aligned}
\norm{v}_X \leq& C\left(\norm{f_1 +\nabla\cdot(a_1\nabla \tilde{u}_1)-b_1 \tilde{u}_1}_{L^2(\Omega_1)} + \norm{f_2 +\nabla\cdot(a_2\nabla \tilde{u}_2)-b_2 \tilde{u}_2 }_{L^2(\Omega_2)}\right)\\
\leq & C\left(\norm{f_1}_{L^2(\Omega_1)}+ \norm{f_2}_{L^2(\Omega_2)}+\norm{\tilde{u}_1}_{H^2(\Omega_1)} + \norm{\tilde{u}_2}_{H^2(\Omega_2)}\right).
\end{aligned}
\end{equation*}
Finally, we obtain that $u = v+\tilde{u}$ solves the problem (\ref{eq:interface problem}) and that
\begin{equation*}
\begin{aligned}
\norm{u}_X \leq& \norm{v}_X + \norm{\tilde{u}}_X \\
\leq& C\left(\norm{f_1}_{L^2(\Omega_1)}+ \norm{f_2}_{L^2(\Omega_2)}+\norm{\tilde{u}}_X\right)\\
\leq &C\left(\norm{f_1}_{L^2(\Omega_1)}+ \norm{f_2}_{L^2(\Omega_2)}+\norm{g}_{H^{3/2}(\partial\Omega)} + \norm{\varphi}_{H^{3/2}(\Gamma)}+ \norm{\psi}_{H^{1/2}(\Gamma)}\right)\\
\leq &C\left(\norm{f_1}_{L^2(\Omega_1)}+ \norm{f_2}_{L^2(\Omega_2)}+\norm{g}_{H^{2}(\partial\Omega)} + \norm{\varphi}_{H^{2}(\Gamma)}+ \norm{\psi}_{H^{1}(\Gamma)}\right),
\end{aligned}
\end{equation*}
where $C$ is a generic constant. The proof is completed.
\end{proof}

With these results, we are able to provide the proof of Theorem \ref{thm:main-elliptic}.
\begin{proof}[proof of Theorem \ref{thm:main-elliptic}]
Lemma \ref{lem:bound of solu} implies the existence and the uniqueness of solution $u^{*}$. Let $$(u_{1, \bm{m}}, u_{2, \bm{m}}) \in (\mathcal{H}_{1,\bm{m}},\mathcal{H}_{2,\bm{m}})$$ be a minimizer of the Lipschitz regularized loss $\text{Loss}_{\bm{m}}(\cdot;\bm{\lambda}, \bm{\lambda}_{\bm{m}}^R)$ \eqref{def: GE LIP PINN empirical loss}. Again, by Lemma \ref{lem:bound of solu}, we have that
\begin{equation*}
\begin{aligned}
\norm{u_{1, \bm{m}} - u^*}_{H^2(\Omega_1)} + \norm{u_{2, \bm{m}} - u^*}_{H^2(\Omega_2)} \leq C \Big(&\norm{\mathcal{L}_1[u_{1,m}] - f_1}_{L^2(\Omega_1)} +
\norm{\mathcal{L}_2[u_{2,m}] - f_2}_{ L^2(\Omega_2)}\\
&+\norm{u_{2,m} - g}_{H^2(\partial\Omega)}
+\norm{u_{2,\bm{m}} -u_{1,\bm{m}} -\varphi}_{H^2(\Gamma)}\\
&+
\norm{a\nabla u_{2,\bm{m}}\cdot \mathbf{n} - a\nabla u_{1,\bm{m}}\cdot \mathbf{n}- \psi}_{H^2(\Gamma)}
\Big).
\end{aligned}
\end{equation*}

Finally, by the second term of Lemma \ref{lem:conv-loss}, we conclude that with probability one over iid samples,
\begin{equation*}
\lim_{m_{r_1} \to \infty} u_{1, \bm{m}} = u^*,\quad \text{in } H^2(\Omega_1), \qquad \lim_{m_{r_1} \to \infty} u_{2, \bm{m}} = u^*,\quad \text{in } H^2(\Omega_2),
\end{equation*}
which completes the proof.
\end{proof}

\section{Summary}\label{summary}
The main contribution of this paper is to perform the convergence analysis of the neural network method for solving linear second-order elliptic interface problems. It is proved that the neural network sequence converges to the unique solution to the interface problem in $H^2$. Numerical results are presented to show agreement with the theoretical findings. This result advanced the mathematical foundations of the deep learning-based solver of PDEs.

To complete the proof, we first derive a Lipschitz regularized empirical loss from the probabilistic space filling arguments \cite{calder2019consistency} to bound the expected PINN loss and then show that the expected PINN loss at the minimizers of the Lipschitz regularized empirical loss converges to zero. Finally, we demonstrate that the minimizers of the Lipschitz regularized empirical losses converge to the solution to the interface problem uniformly as the number of training samples grows in $H^2$ and conclude the main theorem.

There are several interesting further research directions. The landscape of non-convex objective functions and the stochastic gradient optimization process remain open. We would like to further investigate such problems and quantify the optimization error of solving elliptic interface problems using neural networks in the future. In addition, further error analysis to provide a more restrictive error bound is another interesting direction.

\section*{Acknowledgments}
The author would like to thank Professor Hehu Xie for valuable discussions. This work was supported by the National Natural Science Foundation of China (Grant Nos. 11771435, 22073110 and 12171466).


\bibliography{main}
\end{document}